\documentclass[11pt]{amsart}
\allowdisplaybreaks
\usepackage{amsrefs} % Required for references
\usepackage{fancyhdr}
\usepackage[titletoc]{appendix}

\usepackage{amsmath,amssymb,amsthm}
\usepackage{bm}% bold math
\usepackage[dvipsnames]{xcolor}
\usepackage{etoolbox,adjustbox}
\usepackage{abstract}

\usepackage[overload]{textcase}
\usepackage{multirow}
\usepackage{graphics,graphicx}% Include figure files
\usepackage{float,verbatim,array}
\usepackage{todonotes}
\usepackage{xfrac,enumitem,soul}
\usepackage{tikz,tikz-cd}
\usepackage[breaklinks=true]{hyperref}
\usepackage[top=1in, bottom=1in, left=1in, right=1in]{geometry}
\usepackage{color,xcolor}
\colorlet{BLUE}{blue}
\colorlet{RED}{red}
\colorlet{GRAY}{gray}
\colorlet{BROWN}{brown}
\definecolor{OliveGreen}{rgb}{0,0.6,0}

\counterwithout{equation}{section} % undo numbering system provided by phstyle.cls
\numberwithin{equation}{section}
\newtheorem{thm}{Theorem}[section]
\newtheorem{lmm}[thm]{Lemma}%[section]
\newtheorem{crl}[thm]{Corollary}%[section]
\newtheorem{prp}[thm]{Proposition}%[section]

\theoremstyle{definition}
\newtheorem{dfn}[thm]{Definition}%[section]
\newtheorem{rmk}[thm]{Remark}%[section]
\newtheorem{question}[thm]{Question}%[section]
\theoremstyle{definition}
\newtheorem*{thm*}{Theorem}
\newtheorem*{lmm*}{Lemma}
\newtheorem*{crl*}{Corollary}
\newtheorem*{MRI*}{Result 1}
\newtheorem*{MRII*}{Result 2}
\newtheorem*{PRT*}{Path representation theorem}
\newtheorem*{claim*}{Claim}

\title[Completeness of topological spaces]{Completeness of topological spaces: An induction-free review}
\author{Earnest Akofor}
\address{\textnormal{Department of Mathematics and Computer Science, %\\~\\
         Faculty of Science, University of Bamenda, %\\~\\
         PO Box 39 Bambili, NW Region, Cameroon}}
\email{eakofor@gmail.com}

\begin{document}

\begingroup
\def\uppercasenonmath#1{} % this disables uppercasing title
\let\MakeUppercase\relax % this disables uppercasing authors
\maketitle
%\tableofcontents
%%%%%%%%%%%%%%%%%%%%%%%%%%%%%%%%%%%%%%%%%%%%%%%%%%%%%%%%%%%%%%%%%%%%%%%%%%%

\begin{abstract}    % type your abstract below
\vspace{0.2cm}

\noindent Completeness for a (topological) space is often based on the existence of special structures (such as metrics, uniformities, proximities, convergences, etc) that explicitly induce the topology, making the completeness induction-dependent. However, in any given space $X=(X,\tau)$, suppose we fix a base $\mathcal{B}$ of $\tau$ that is \emph{graded}, in the sense it is partitioned as $\mathcal{B}=\bigcup_{\varepsilon\in \mathcal{E}}\mathcal{B}_\varepsilon$ into open covers $\mathcal{B}_\varepsilon$ of $X$, making $X=(X,\tau,\mathcal{B})$ a \emph{(graded) base space}. If we now relax the notion of \emph{convergence of nets} to a notion of \emph{approach between nets} in $X$, then we obtain a more natural \emph{induction-free} notion of a \emph{cauchy net} in a base space, hence a corresponding \emph{induction-free} notion of \emph{completeness} for base spaces. We find that many classical concepts and results on completeness for uniform spaces carry over to completeness for a certain class of base spaces (named \emph{locally symmetric base spaces} or \emph{$lsb$-spaces}) that properly contains uniform spaces. The said classical results include characterization of compactness, Baire's theorem, existence of a completion, and completeness results for product and function $lsb$-spaces.
\end{abstract}
%%%%%%%%%%%%%%%%%%%%%%%%%%%%%%%%%%%%%%%%%%%%%%%%%%%%%%%%%%%
%%%%%% Unbold TOC text and page numbers %%%%%%%%%%%%%%%%%%%
\let\bforigdefault\bfdefault
\addtocontents{toc}{\let\string\bfdefault\string\mddefault}
%%%%%%%%%%%%%%%%%%%%%%%%%%%%%%%%%%%%%%%%%%%%%%%%%%%%%%%%%%%
\tableofcontents
%\maketitle
%%%%%%%%%%%%%%%%%%%%%%%%%%%%%%%%%%%%%%%%%%%%%%%%%%%%%%%%%%%
%%%%%%%%%%%%%%%%%%%%   Start of main body of article
%%########################################################################################################################
%%########################################################################################################################
\section{\textnormal{\bf Introduction}}\label{Intro}
%\begin{comment}
%%########################################################################################################################
%%########################################################################################################################
In our discussion, the words ``function'' and ``map'' are synonymous.
\begin{center}
\emph{Motivation and related work}
\end{center}

In a metric space $X=(X,d)$ (resp., topological group $X=(X,-)$), the metric $d$ (resp., the algebraic difference $-:(g,h)\mapsto gh^{-1}$) allows us to define the notion of a cauchy net, which is then used to define completeness for the space. More generally, in a uniform space $(X,\mathcal{U})$, the uniform structure $\mathcal{U}$ provides a system of ball-like neighborhoods similar to that provided by the metric in a metric space, leading to an analogous notion of a cauchy net, which is then used to define completeness (see Definition \ref{uCmpltnDfn}). Other spaces with similar structures (to the metric, algebraic difference, uniform structure) and their generalizations include proximity spaces \cite{NmpWrrk1970}, cauchy spaces \cite{Schect1997,Lowen1989}, convergence spaces \cite{BentEtal1990,Dolecki2024,DolMyn2016,Choquet47-48}, connector spaces \cite{Nakano1977}. In all these situations, the structures used to define cauchyness are the same structures that induce the underlying topology of the space, making the notion of cauchyness induction-dependent. ({\bf footnote}\footnote{
There of course exist many other generalizations of completeness for topological spaces, but our study is limited to those directly based on the original notion of a cauchy sequence.
}).

In a generic (topological) space $X=(X,\tau)$ however, we would like to have an induction-free notion of a cauchy net, which would then give an induction-free notion of completeness. Providing such a notion, and some immediate consequences, is our main goal in this paper. A review of preliminary concepts and results related to our discussion can be found in many texts on general topology, including \cite[Ch.s 7,8,12]{James1987}, \cite[Ch. 6]{kelley1955}, \cite[Ch. 8]{engelking1989}, \cite[Ch.s II,IX,X]{BbkGT1966}, \cite[Ch. 3]{NmpWrrk1970}, \cite{Isbell1964}, \cite{neumann1935}, and the references therein.

Given a space $X=(X,\tau)$, suppose we fix a base $\mathcal{B}$ of $\tau$ that is \textbf{graded}, in the sense that it is partitioned as $\mathcal{B}=\bigcup_{\varepsilon\in \mathcal{E}}\mathcal{B}_\varepsilon$ into open covers $\mathcal{B}_\varepsilon$ of $X$, making $X=(X,\tau,\mathcal{B})$ a \textbf{(graded) base space}. If we now relax the usual notion of \emph{convergence of nets} to a notion of \emph{approach between nets} (Definition \ref{ApprchDfn}), then we obtain a more natural induction-free notion of a \textbf{cauchy net} in a base space, hence a corresponding induction-free notion of \textbf{completeness} (Definition \ref{ComplDfn}). ({\bf footnote}\footnote{
One could argue that fixing a graded base is similar to using structures that induce the topology. This is partly true since the usual topology inducing structures such as metrics and uniform structures can be achieved using graded bases (see the notion of $u$-spaces in Definition \ref{uCmpltnDfn}). However, given a space $(X,\tau)$ and a base $\mathcal{B}\subset\tau$, picking a grading for $\mathcal{B}$ is in general independent of (hence should not affect) the nature of $\tau$, and the associated definition of approach between nets (Definitions \ref{NtApSumrDfn}, \ref{NtApSumrDfn2}, \ref{ApprchDfn}) is relaxed/flexible enough to admit a large class of other equally valid induction independent notions of completeness. Therefore, just as we have a notion of ``convergence'' that requires no prior restriction on the topology of the space, we also have a notion of ``cauchyness'' that requires no prior restriction on the topology of the space.
}). We also consider an induction-free notion of a \textbf{topology of uniform convergence} $\tau_{uc}$ (Definition \ref{TpUnfCnvDfn}) on function spaces $X^Y$ (for sets $Y$). Our discussion, which is mainly a review of well known classical results (Theorems \ref{CompCsbThm}, \ref{CpctSpThm}, \ref{BaireThm}, \ref{tUConThm}, \ref{BSpCmplThm}, \ref{EBPCSthm}, \ref{PtWsCompThm}, \ref{UnfCmplThm}, \ref{ContUlmThm}), closely follows the review of classical results on metric space completeness in \cite[Sec. 25.4.2]{akofor2021}, but as already mentioned, such classical results can be found in many texts on general topology, including especially \cite[Ch.s 7,8,12]{James1987}.

Our general idea is to adopt a straightforward point-set topology oriented viewpoint that moves away from the traditional symmetry assuming ball-like neighborhoods inherent in earlier studies on completeness. In particular, our induction-free notion of cauchyness (Definition \ref{ApprchDfn}) is based on the observation that even in the induction-dependent pioneer studies of completeness, the definition of a cauchy net nonetheless effectively depends only on a choice of \emph{graded base} for the topology, and is independent of how the topology, base, or grading is explicitly described (whether in terms of metrics, algebraic differences, uniformities, proximities, cauchy structures, convergences, connectors, or otherwise) -- see the preliminary discussion in Section \ref{Prelims}. Therefore, the purpose of our view of completeness is to consolidate (i.e., unify and extend) many familiar treatments of completeness while revealing potentially redundant structures and procedures in earlier studies of completeness. Most of our results involve three main classes of base spaces, named in Definition \ref{SymBsSp} as \textbf{$lsb$-spaces} (which strictly contain uniform spaces), \textbf{$csb$-spaces} (which contain uniform spaces), and \textbf{$sb$-spaces}.

Historical reviews of completeness and completion may be found in \cite{BentEtal1990,Dolecki2024,DolMyn2016}.

\begin{center}
\emph{Summary of main results}
\end{center}
In Section \ref{Prelims}, we present a brief unified overview of some familiar treatments of completeness for topological spaces and discuss a framework for extending such treatments based on a notion of approach between nets (Definitions \ref{NtApSumrDfn} and \ref{NtApSumrDfn2}). Next, in Section \ref{ApprSec}, we introduce and analyze a concrete version of the notion of \emph{approach between nets} (Definition \ref{ApprchDfn}). We obtain the following:
\begin{enumerate}
\item {\bf Theorem \ref{CompCsbThm}}: A continuous map between $lsb$-spaces is uniform on compact subsets.
\item {\bf Theorem \ref{CpctSpThm}}:
A $lsb$-space is compact if and only if it is both complete and precompact.
\item {\bf Theorem \ref{BaireThm}}: Baire's theorem holds in a locally complete regular Hausdorff $lsb$-space.
\end{enumerate}

In Section \ref{CmplBasSpSec}, we investigate more implications of \emph{approach between nets} and find that some classical concepts and results on completion for uniform spaces carry over to completion for \emph{$lsb$-spaces}. We obtain the following results:
\begin{enumerate}[resume]
\item {\bf Theorem \ref{tUConThm}}: A uniform map from a dense subset of a Hausdorff $lsb$-space to a complete Hausdorff $lsb$-space has a unique uniform extension.
\item {\bf Theorem \ref{BSpCmplThm}}: Every $lsb$-space admits a completion (which is unique if the space is Hausdorff).
\item {\bf Theorem \ref{EBPCSthm}}: Every $csb$-space is a cauchy space.
\end{enumerate}
In Section \ref{CmPrdFxnSpSec}, we consider completeness for product and function base spaces, along with an induction-free notion of a \emph{topology of uniform convergence} $\tau_{uc}$ on function spaces $X^Y$, and recover the following classical results:
\begin{enumerate}[resume]
\item {\bf Theorem \ref{PtWsCompThm}}: The product $(\prod_{y\in Y}X_y,\tau_p)$ of $lsb$-spaces is complete if and only if each factor $X_y$ is complete.
\item {\bf Theorems \ref{UnfCmplThm}}: If $X$ is a complete $lsb$-space, then so is $(X^Y,\tau_{uc},\mathcal{B})$.
\item {\bf Theorem \ref{ContUlmThm}}: If $X$ is a complete $lsb$-space, then both the space of continuous maps $C(Y,X)\subset (X^Y,\tau_{uc},\mathcal{B})$ and the space of uniformly continuous maps $UC(Y,X)\subset (X^Y,\tau_{uc},\mathcal{B})$ are complete.
\end{enumerate}
 The last two results above remain true if $\tau_{uc}$ is replaced with any topology containing the topology of pointwise convergence $\tau_p$.

In Section \ref{IntModSec}, we indicate how our notion of completeness provides a method of integration in complete topological modules. We conclude in Section \ref{ConclSec} with a summary and questions.

\begin{comment}
%%##############################################
%\newpage
\vspace{0.5cm}
\hrule
\vspace{0.1cm}
\textcolor{red}{To continue revision ...}
\vspace{0.1cm}
\hrule
\vspace{0.5cm}
%%##############################################
\end{comment}

%%########################################################################################################################
%%########################################################################################################################
\section{\textnormal{\bf Preliminary observations and net-approach structures}}\label{Prelims}
%\begin{comment}
%%########################################################################################################################
%%########################################################################################################################
%\section{General completeness (Combine this with fact that every space is ``somehow metrizable'' to generalize past published work)}
\noindent In this section, we introduce $u$-structures as convenient generalizations of uniform structures. Next, we introduce net-approach structures generalizing convergence structures. Finally, we further generalize $u$-spaces to base spaces which can support a concrete net-approach structure (to be introduced in the next section).

\begin{dfn}[{$u$-structure, $u$-space, $u$-cauchyness, $u$-completeness}]\label{uCmpltnDfn}
As in \cite[Definition 25.1.5]{akofor2021}, let $X=(X,\tau)$ and $Z=(Z,\tau_Z)$ be spaces and $~u:(X\times X,\tau\times\tau)\rightarrow Z~$ a continuous map (call it a \textbf{$u$-structure} on $X$). A net $s:I\rightarrow X,~i\mapsto s_i$ is \textbf{$u$-cauchy} in $X$ if the associated bi-net $u\circ(s\times s):I\times I\stackrel{s\times s}{\longrightarrow}X\times X\stackrel{u}{\longrightarrow}Z,~(i,j)\mapsto u(s_i,s_j)$ converges in $Z$ to a point in the set $u(\{(x,x):x\in X\})$. The space $X$ is \textbf{$u$-complete} if every $u$-cauchy net in $X$ converges. ({\bf footnote}\footnote{
More generally, given nets $s:I\rightarrow X,~i\mapsto s_i$ and $t:J\rightarrow X,~j\mapsto t_i$, write $s\sim t$ if the associated bi-net $u\circ(s\times t):I\times I\stackrel{s\times t}{\longrightarrow}X\times X\stackrel{u}{\longrightarrow}Z,~(i,j)\mapsto u(s_i,t_j)$ converges in $Z$ to a point in the set $u(\{(x,x):x\in X\})$. Then $\sim$ is an equivalence relation on cauchy nets (provided $u$ is accordingly chosen/constructed), where we say $s$ is \textbf{$u$-cauchy} if $s\sim s$ (which then determines when $X$ is \textbf{complete} as usual). A \textbf{completion} of $(X,\tau)$ is a suitable topologization $([X],[\tau])$ of the set of equivalence classes $[X]:=\{[s]:\textrm{$s$ a cauchy net in $X$}\}$ (where $[s]:=\{t:s\sim t\}$) containing $X$ as a uniformly embedded dense subspace.
}). If we fix a collection of open sets $\mathcal{E}\subset\tau_Z$ in $Z$ (where $\mathcal{E}$ is closed under finite intersections), then the \textbf{$u$-ball} in $X$ of ``radius'' $\varepsilon\in\mathcal{E}\backslash\{\emptyset\}$ centered at $x\in X$ is
\begin{align}
\label{uBallEqn}B_\varepsilon(x)=B^u_\varepsilon(x):=\{y\in X:u(x,y)\in \varepsilon\}.~~(\textrm{\bf footnote}\footnotemark).
\end{align}
\footnotetext{
This is to be compared with the concept of ``set-valued metrics'' in \cite{akofor25svm}, where the condition $u(x,y)\in\varepsilon$ or $\{u(x,y)\}\subset\varepsilon$ (here) corresponds to the condition $d_{sv}(x,y)\prec\varepsilon$ (in \cite{akofor25svm}), for a set-valued metric
\[
d_{sv}:X\times X\rightarrow \mathcal{E}\subset\mathcal{P}(Z).
\]
}The map $u$ induces a topology on $X$ given by
\[
\tau_u:=\{O\subset X:\textrm{$\forall x\in O$, $\exists \varepsilon\in\mathcal{E}~$ s.t. $B_\varepsilon(x)\subset O$}\},~~\textrm{({\bf footnote}\footnotemark),}
\]\footnotetext{
\textbf{Showing $\tau_u$ is a topology:}  (i) It is clear that $\emptyset,X\in\tau_u$. (ii) It is also clear $\tau_u$ is closed under arbitrary unions. (iii) If $O_1,O_2\in\tau_u$, then for any $x\in O_1\cap O_2$, some $B_{\varepsilon_i}(x)\subset O_i$, and so $x\in B_{\varepsilon_1\cap\varepsilon_2}(x)=B_{\varepsilon_1}(x)\cap B_{\varepsilon_2}(x)\subset O_1\cap O_2$, hence so $O_1\cap O_2\in\tau_u$ (provided $\mathcal{E}$ is closed under finite intersections, so that $\varepsilon_1\cap\varepsilon_2\in\mathcal{E}$).
}and we call $X$ a \textbf{$u$-space} if $\tau=\tau_u$ (written $X=(X,u)$). \textbf{Product $u$-spaces} are discussed in Definition \ref{PrdUspDfn}. If we further assume $u$ and $\mathcal{E}$ are such that for any $\varepsilon\in\mathcal{E}$, there exists $\varepsilon'\in \mathcal{E}$ satisfying $u^{-1}(\varepsilon')\circ u^{-1}(\varepsilon')\subset u^{-1}(\varepsilon)$ and $u(u^{-1}(\varepsilon')\circ u^{-1}(\varepsilon'))\in\mathcal{E}$, ({\bf footnote}\footnote{
The composition of binary relations $R_1,R_2$ is given by ~$R_1\circ R_2:=\{(a,b):\textrm{$\exists c$ with $(a,c)\in R_1$ and $(c,b)\in R_2$}\}$.
}), then the $\tau_u$-interior of a set $A\subset X$ satisfies
\[
A^o=\{x\in X:~\textrm{some}~B_\varepsilon(x)\subset A\}.~\textrm{({\bf footnote}\footnotemark)}.
\]
\footnotetext{
Indeed, let $O:=\{x\in X:~\textrm{some}~B_\varepsilon(x)\subset A\}$ and fix $x\in O$. Let $B_\varepsilon(x)\subset A$, $u^{-1}(\varepsilon')\circ u^{-1}(\varepsilon')\subset u^{-1}(\varepsilon)$, and $y\in B_{\varepsilon'}(x)$. Then for any $z\in B_{\varepsilon'}(y)$, we have $u(x,y),u(y,z)\in \varepsilon'$, i.e., $(x,y),(y,z)\in u^{-1}(\varepsilon')$, which implies $(x,z)\in u^{-1}(\varepsilon')\circ u^{-1}(\varepsilon')$, and so $z\in B_{u(u^{-1}(\varepsilon')\circ u^{-1}(\varepsilon'))}(x)\subset B_\varepsilon(x)$. Therefore $B_{\varepsilon'}(y)\subset A$, which implies $y\in O$, and so $B_{\varepsilon'}(x)\subset O$. This proves that $O=A^o$ (as the largest $\tau_u$-open set contained in $A$).}
\end{dfn}

Several well known versions of completeness for topological spaces are special cases of $u$-completeness. In particular, for (i) a metric space $(X,d)$, (ii) a topological group $(X,-)$, (iii) a uniform space $(X,\mathcal{U})$, the associated maps $u:X\times X\rightarrow Z$ are respectively given by
\begin{enumerate}
\item[(i)] $u:(X\times X,\tau\times\tau)\rightarrow ([0,\infty),\mathcal{E}),~(x,y)\mapsto d(x,y)$, {~} $\mathcal{E}:=\{[0,r):r>0\}$
\item[(ii)] $u:(X\times X,\tau\times\tau)\rightarrow (X,\mathcal{E}),~(x,y)\mapsto xy^{-1}$, {~} $\mathcal{E}:=\{Q\in\tau:id\in Q\}$
\item[(iii)] $u:(X\times X,\tau\times\tau)\rightarrow (X\times X,\mathcal{E}),~(x,y)\mapsto (x,y)$, {~} $\mathcal{E}:=\mathcal{U}$.
\end{enumerate}

These $u$-based methods of completeness seem to suggest that the map $u$ itself is essential for the definition of a cauchy net in a given topological space. This observation is also supported by the existence of \emph{cauchy spaces} introduced as \emph{convergence spaces} with an automatic/built-in notion of a cauchy net (hence generalizing \emph{uniform spaces}), \cite{keller1968}. Moreover, in the $u$-based methods, convergence and cauchyness seem to require very different treatments, with convergence characterized independently of the way the topology is induced by $u$, while cauchyness (through its dependence on $u$) depends on the way the topology is induced by $u$.

However, we want to adopt a more natural treatment of completeness such that: (i) both the topology of the space and the notion of a cauchy net in the space do not explicitly require existence of the map $u$, and (ii) both \emph{convergence} and \emph{cauchyness} of nets emerge from the same operation, namely, \emph{approach} between nets (Definitions \ref{NtApSumrDfn} and \ref{ApprchDfn}). The basic idea is to obtain cauchyness from an induction-free relation (between nets) that restricts to convergence on constant nets, as summarized in Definition \ref{NtApSumrDfn} (a concrete realization of which is detailed in Sections \ref{ApprSec} and \ref{CmplBasSpSec}, through Definition \ref{ApprchDfn}, and concluding in Theorem \ref{NtAppSpThm} that a base space is a net-approach space).

The following two definitions use notation (for subnets) from the introduction of nets in Section \ref{ApprSec}. ({\bf footnote}\footnote{
In addition to Definitions \ref{NtApSumrDfn} and \ref{NtApSumrDfn2}, we may define a more primitive net-approach structure/space as follows:  Let $X$ be a set and $N[X]$ the class of nets on $X$. A relation $R\subset N[X]\times N[X]$ is a \textbf{net-approach structure} (in which case $uRv$ is written as $u\rightarrowtail v$), making $X=(X,R)$ a \textbf{net-approach space}, if it satisfies the following:
\begin{enumerate}[leftmargin=0.7cm]
\item\label{AbNetApA1} $R$ is (left) hereditary: If $u\rightarrowtail v$, then $u\circ\phi \rightarrowtail v$ for every subnet $u\circ\phi$.
\item\label{AbNetApA2} $R$ is transitive: If $u\rightarrowtail v$ and $v\rightarrowtail w$, then $u\rightarrowtail w$.
\item\label{AbNetApA3} $R$ is preserved by marginal convergence: If $u_{\bullet\bullet}\rightarrowtail v_{\bullet\bullet}$, $u_{\bullet\beta}\rightarrowtail u'_\beta$ (for each $\beta$), and $v_{\bullet j}\rightarrowtail v'_j$ (for each $j$), then $u'_\bullet\rightarrowtail v'_\bullet$.
\end{enumerate}
Any \emph{net-approach structure} induces a \emph{net-convergence structure} \cite{ObrienEtal2021,OBrien2021}, and this is clear from Definitions \ref{NtApSumrDfn} and \ref{NtApSumrDfn2}.
}).

\begin{dfn}[{Net-approach structure, Net-approach space}]\label{NtApSumrDfn}
Let $X=(X,\tau)$ be a space and $N(X)$ a \textbf{sufficient set of nets} in $X$ (i.e., the topological properties of $X$ can be fully described in terms of nets in $N(X)$) similar to that constructed in \cite{ObrienEtal2021,OBrien2021} for example. A relation $R\subset N(X)\times N(X)$ is a \textbf{net-approach structure} (in which case $uRv$ is written as $u\rightarrowtail v$), making $X=(X,\tau,R)$ a \textbf{net-approach space}, if it satisfies the following:
\begin{enumerate}[leftmargin=0.7cm]
\item\label{NetAppAx1} $R$ restricts to convergence on constant nets: For any $u\in N(X)$, $u\rightarrowtail\{x\}$ iff $u\rightarrow x$ (where $u\rightarrowtail\{x\}$ is also written as $u\rightarrowtail x$).
%%\item\label{NetAppAx2} $R$ is locally symmetric (or symmetric on convergent nets): For any $u\in N(X)$, if $u\rightarrowtail\{x\}$, then $\{x\}\rightarrowtail u$.
\item\label{NetAppAx2} $R$ is (left) hereditary: If $u\rightarrowtail v$, then $u\circ\phi \rightarrowtail v$ for every subnet $u\circ\phi$.
\item\label{NetAppAx3} $R$ is transitive: If $u\rightarrowtail v$ and $v\rightarrowtail w$, then $u\rightarrowtail w$.
\item\label{NetAppAx4} $R$ is preserved by marginal convergence: If $u_{\bullet\bullet}\rightarrowtail v_{\bullet\bullet}$, $u_{\bullet\beta}\rightarrowtail u'_\beta$ (for each $\beta$), and $v_{\bullet j}\rightarrowtail v'_j$ (for each $j$), then $u'_\bullet\rightarrowtail v'_\bullet$.
\end{enumerate}
A net-approach space will be called
\begin{itemize}[leftmargin=0.7cm]
\item \textbf{locally symmetric} if $R$ is symmetric on convergent nets, in that for any $u\in N(X)$ and $x\in X$, if $u\rightarrowtail\{x\}$, then $\{x\}\rightarrowtail u$.
\item \textbf{cauchy symmetric} if $R$ is cauchy symmetric on cauchy nets, in that for any $u,v\in N(X)$, if $v$ is cauchy and $u\rightarrowtail v$, then $v\rightarrowtail u$.
\item \textbf{symmetric} if $R$ is symmetric, i.e., for any $u,v\in N(X)$, if $u\rightarrowtail v$, then $v\rightarrowtail u$.
\end{itemize}
A net $u\in N(X)$ is \textbf{cauchy} (or \textbf{$R$-cauchy}) if, for all subnets $u\circ\phi$ and $u\circ\psi$, we have $u\circ\phi\rightarrowtail u\circ\psi$ (i.e., all subnets of $u$ approach each other). In a locally symmetric net-approach space, it is clear that \textbf{a convergent net is cauchy}. A net-approach space $X$ is \textbf{complete} if every cauchy net in $X$ converges, and a \textbf{completion} of $X$ is a complete net-approach space containing $X$ as a uniformly embedded dense subset. A map between net-approach spaces is \textbf{continuous} (resp., \textbf{cauchy continuous}) if it maps convergent nets to convergent nets (resp., cauchy nets to cauchy nets), and \textbf{uniformly continuous} if it maps approaching nets to approaching nets.
\end{dfn}

In the following section, we will construct (via Theorem \ref{NtAppSpThm}) a concrete net-approach relation (Definitions \ref{ApprchDfn},\ref{CovNetDfn}) on \emph{base spaces} (Definition \ref{LocSpDfn}).

Our \emph{net-approach} method, as summarized in Definition \ref{NtApSumrDfn}, may also be extended to a (net) convergence space, say using the notation of \cite{ObrienEtal2021,OBrien2021}, as follows.

\begin{dfn}[{NAC structure, NAC space}]\label{NtApSumrDfn2}
Let $X=(X,c)$ be a \emph{net convergence space} (with $c$ a \emph{net convergence structure} ({\bf footnote}\footnote{
As in \cite{ObrienEtal2021,OBrien2021}, a \textbf{net convergence structure} on $X$ is a map $c:X\rightarrow \mathcal{P}(N(X))$ such that:
\begin{enumerate}
\item Every constant net converges, i.e., $\{x\}\in c(x)$ or $\{x\}\stackrel{c}{\longrightarrow}x$, for each $x\in X$.
\item If a tail of $u$ is a subnet of $v$ (written $u\preceq v$) and $v\stackrel{c}{\longrightarrow}x$, then $u\stackrel{c}{\longrightarrow}x$.
\item If $u$ and $v$ have the same domain, $u\stackrel{c}{\longrightarrow}x$, $v\stackrel{c}{\longrightarrow}x$, and a net $w$ satisfies $w_\alpha\in\{u_\alpha,v_\alpha\}$ for each $\alpha$, then $w\stackrel{c}{\longrightarrow}x$.
\end{enumerate}
}))
and $N(X)$ a sufficient set of nets in $X$. A \textbf{net-approach convergence structure} (or \textbf{NAC structure}) on $X$ is a map $\rho:N(X)\rightarrow\mathcal{P}(N(X)),~u\mapsto\rho(u)$ (making $X=(X,c,\rho)$ a \textbf{NAC space}) with the properties below (which are based on the four axioms in Definition \ref{NtApSumrDfn}). The structure $\rho$ simply specifies, for each $v\in N(X)$, the set $\rho(v)\subset N(X)$ of all nets $u$ that \textbf{approach} $v$ through $\rho$, written $u\in\rho(v)$ (or $u\stackrel{\rho}{\rightarrowtail}v$).
\begin{enumerate}[leftmargin=0.7cm]
\item\label{NetCAppAx1} Approach restricts to convergence on constant nets ($\rho|_{ConstN(X)}=c$): A net $u$ approaches a constant net $\{x\}$ iff $u$ converges to $x$, i.e., $u\in\rho(\{x\})$ $\iff$ $u\stackrel{c}{\longrightarrow}x$.
%%\item\label{NetCAppAx2} Approach is locally symmetric (or symmetric on convergent nets): For every convergent nets $u$, $u\in\rho(\{x\})$ $\Rightarrow$ $\{x\}\in\rho(u)$.
\item\label{NetCAppAx2} Approach is transitive: If a net $u$ approaches a net $v$ which in turn approaches a net $w$, then $u$ approaches $w$, i.e., $u\in\rho(v)$ and $v\in\rho(w)$ $\Rightarrow$ $u\in\rho(w)$.
\item\label{NetCAppAx3} Approach is (left) hereditary: If a net $u$ approaches a net $v$, then every subnet of $u$ approaches $v$, i.e., $u\in\rho(v)$ $\Rightarrow$ $u\circ\phi\in\rho(v)$ for every subnet $u\circ\phi$.
\item\label{NetCAppAx4} Approach is preserved by marginal convergence: If $u_{\bullet\bullet}\in\rho(v_{\bullet\bullet})$, $u_{\bullet\beta}\in \rho(u'_\beta)$ (for each $\beta$), and $v_{\bullet j}\in\rho(v'_j)$ (for each $j$), then $u'_\bullet\in\rho(v'_\bullet)$.
\end{enumerate}
As before, we can define a \textbf{locally symmetric NAC space} (via $u\in\rho(\{x\})$ $\Rightarrow$ $\{x\}\in\rho(u)$), a \textbf{cauchy symmetric NAC space} (via $v$ cauchy and $u\in\rho(v)$ $\Rightarrow$ $v\in\rho(u)$), and a \textbf{symmetric NAC space} (via $u\in\rho(v)$ $\Rightarrow$ $v\in\rho(u)$). A net $u\in N(X)$ is \textbf{cauchy} (or \textbf{$\rho$-cauchy}) if all subnets of $u$ approach each other (i.e., $u\circ\phi\stackrel{\rho}{\rightarrowtail}u\circ\psi$ for all $\phi,\psi$). A NAC space $(X,c,\rho)$ is \textbf{complete} if every cauchy net in $(X,c)$ converges, and a \textbf{completion} of $(X,c,\rho)$ is a complete NAC space containing $(X,c,\rho)$ as a uniformly embedded dense subset. A map $f:X\rightarrow Y$ between NAC spaces is \textbf{continuous} (resp., \textbf{cauchy continuous}) if it maps convergent nets to convergent nets (resp., cauchy nets to cauchy net), and \textbf{uniformly continuous} if it maps approaching nets to approaching nets. The \textbf{convergence topology} (if it exists, making $c$ a \textbf{topological convergence}
-- {\bf footnote}\footnote{
Topological convergence admits the property that any net satisfying ``\emph{every subnet has a convergent subnet}'' must converge. Since convergence almost everywhere (CAE) has non-convergent nets with this property, CAE is not a topological convergence.
}) on $X=(X,c)$ induced by $c$ is
\[
\tau_c:=\{O\subset X:\textrm{every net in $X\backslash O$ that converges in $X$ converges to a point in $X\backslash O$}\}.
\]
\end{dfn}

We will now introduce the basic concepts underlying our concrete net-approach structure to be introduced in the next section.

\begin{dfn}[{Base space}]\label{LocSpDfn}
Let $X=(X,\tau)$ be a space. A \textbf{(graded) base space} of $X$ is a triple $X=(X,\tau,\mathcal{B})$, for a (sub)base $\mathcal{B}\subset\tau$ with a \textbf{grading} $\mathcal{B}=\bigcup_{\varepsilon\in\mathcal{E}}\mathcal{B}_\varepsilon$ for open covers $\mathcal{B}_\varepsilon$ of $X$ (where a subset of $\mathcal{B}_\varepsilon$ is called a \textbf{homogeneous subset} of $\mathcal{B}$). ({\bf footnote}\footnote{
More generally, we can replace the base space with a \textbf{cover space} $X=(X,\tau,\mathcal{C})$, where $\mathcal{C}=\{\mathcal{C}_\varepsilon\}_{\varepsilon\in\mathcal{E}}$ is a family of open covers of $X$. Definition \ref{ApprchDfn} generalizes easily to cover spaces (with $\mathcal{C}_\varepsilon$ playing the role of $\mathcal{B}_\varepsilon$). However, we might need to settle for a new notion of convergence, since convergence wrt $\mathcal{C}$ (unlike that wrt $\mathcal{B}$) might not be equivalent to convergence wrt $\tau$.
}). We note that in Definition \ref{uCmpltnDfn}, $\mathcal{E}$ was a specific collection of open subsets of the auxiliary space $Z=(Z,\tau_Z)$, but here $\mathcal{E}$ is an arbitrary index set.

\end{dfn}
Every uniform space $X=(X,\mathcal{U})$ is associated the \textbf{standard uniform base space} $(X,\tau_\mathcal{U},\mathcal{B}_\mathcal{U})$, where $\tau_\mathcal{U}:=\{O\subset X:\forall x\in O,~\exists U\in\mathcal{U}~\textrm{s.t.}~B_U(x)\subset O\}$ is the $\mathcal{U}$-induced \textbf{uniform topology} and $\mathcal{B}_\mathcal{U}=\bigcup_{U\in\mathcal{U}}\mathcal{B}_U$ is the \textbf{graded standard base} of $\tau_\mathcal{U}$, with $\mathcal{B}_U$ the cover of $X$ consisting of open balls
\[
\textrm{$B_U(x)^o=\{x'\in X:\textrm{some}~B_V(x')\subset B_U(x)\}$, for $U\in\mathcal{U}$ and $x\in X$},
\]
where ~$B_U(x):=\{y\in X:(x,y)\in U\}$~ as in Equation (\ref{uBallEqn}). Subsequently, we will for convenience ignore the interior symbol $^o$ on the balls.

\begin{rmk}
Subsequently, every uniform space (e.g., a metric space) will by default be viewed as its standard uniform base space (i.e., the base space with respect to the base of open balls).
\end{rmk}

\begin{dfn}[{Filter}]
Let $X$ be a set. A \textbf{filter} in $X$ is a collection $\mathcal{F}\subset\mathcal{P}(X)$ of subsets of $X$ with these properties: (i) $\emptyset\not\in\mathcal{F}$; (ii) If $A,B\in\mathcal{F}$, then $A\cap B\in\mathcal{F}$; (iii) If $A\in\mathcal{F}$, then $B\in\mathcal{F}$ for all $B\supset A$.

Although it is not clear whether or not our discussion can be rephrased entirely in terms of filters (Question \ref{AprDfnQsn}), there is a close relationship between nets an filters (see \cite[Ch. IX]{Freiwald2014}, \cite{Bartle1955}) by which it is expected that most discussions in terms of nets can be expressed in terms of filters and vice versa. Filters will be strictly required only in a few instances, including Definition \ref{ConvChSpDfn} and Theorem \ref{EBPCSthm} to prove that a $csb$-space (Definition \ref{SymBsSp}) is a cauchy space. Nevertheless, as indicated in \cite{Bartle1955}, a study that flexibly employs both nets and filters interchangeably is most convenient, since nets and filters seem to be naturally complementary.

%% In fact, some of our nets based proofs (e.g., the proofs of Lemma \ref{tUConLmm4} and Theorem \ref{tUConThm} via Lemmas \ref{NFBscsLmm} and \ref{CchyFltLmm}) use knowledge of the close connection between nets and filters.
\end{dfn}

Throughout, the phrases ``cauchy'', ``complete'', ``uniform'', ``precompact'' should be prefixed by ``\emph{induction-free}'', but for notational convenience, we will drop the prefix.

%%########################################################################################################################
\section{\textnormal{\bf Concrete net-approach and completeness in base spaces}}\label{ApprSec}
%%########################################################################################################################
%%########################################################################################################################
\noindent
Throughout this section, unless stated otherwise, $X=(X,\tau,\mathcal{B})$ is a base space.

Let $\mathcal{D}\subset\tau$. We refer to an element of $\mathcal{D}$ as a \textbf{$\mathcal{D}$-open set} (or a \textbf{$\mathcal{D}$-neighborhood}) in $X$. Similarly, we refer to an element of $\mathcal{D}^c:=\{O^c:O\in\mathcal{D}\}$ as a \textbf{$\mathcal{D}$-closed set} in $X$. In order to fix notation, we now recall some basic concepts associated with nets.

\begin{dfn}[{Net, Tail, Eventuality (Ultimacy)}]
A \textbf{net}, or \textbf{Moore-Smith sequence}, in $X$ is a map $u=(u_i)_{i\in I}\equiv\{u_i\}_{i\in I}:I\rightarrow X$, $i\mapsto u_i$ on a directed set $I=(I,\leq)$. ({\bf footnote}\footnote{
The directed set here is based on a \emph{pre-order} that need not be a \emph{partial order}. This is because in applications, equality of indices $i$ (or of tails $[i,I]$) of nets $u=\{u_i\}_{i\in I}$ is rarely required.
}). A \textbf{sequence} in $X$ is a net on $\mathbb{N}$.

Let $u:I\rightarrow X$ be a net. Given $i\in I$, let $[i,I]:=\{j\in I:j\geq i\}$ (and call it the \textbf{upper segment} at $i$), which we also write as $[i,~]$ when $I$ is understood, unimportant, or unspecified. The \textbf{tail} of $u$ at $i\in I$ is the restriction $u_{[i,I]}:=u|_{[i,I]}$. A net is \textbf{eventually} (or \textbf{ultimately}) \textbf{constant} if it has a \textbf{constant tail} (i.e., a tail with at most one element). A net is \textbf{nontrivial} if it is not eventually constant (i.e., it has no constant tail).
\end{dfn}

\begin{dfn}[{Subnet, Eventuality, Frequency, Derived filter}]\label{SubnetDfn}
Let $I,J$ be directed sets. A subset $C\subset I$ is a \textbf{cofinal subset} if every element $i\in I$ has an upper bound $c\geq i$, $c\in C$. A map $\phi:I\rightarrow J$ is a \textbf{cofinal map} if for each $j\in J$ there exists $i\in I$ such that $\phi([i,I])\subset[j,J]$ (i.e., $\phi(k)\geq j$ for all $k\geq i$) (in particular, $\phi(I)\subset J$ is a cofinal subset). A net $p:I\rightarrow X$ is a \textbf{subnet} ({\bf footnote}\footnote{
In \cite{AarAnd1972}, given nets $u:I\rightarrow X$ and $v:J\rightarrow X$, we say $u$ is a \textbf{subnet} (call it an \textbf{$AA$-subnet}) of $v$ if for any $j\in J$, there is $i\in I$ such that $u([i,~])\subset v([j,~])$, i.e., every tail of $v$ contains a tail of $u$. Two nets are \textbf{$AA$-equivalent} if they are $AA$-subnets of each other. This alternative definition of a subnet (which is obviously simpler) is both sufficient and more convenient/efficient for expressing the connection between nets and filters \cite{Bartle1955}.
}) of a net $q:J\rightarrow X$ (written $p\subset q$) if there exists a cofinal map $\phi:I\rightarrow J$ such that $~p=q\circ \phi:I\stackrel{\phi}{\longrightarrow}J\stackrel{q}{\longrightarrow}X$. ({\bf footnote}\footnote{Given a nontrivial net $u$, a subnet of $u$ is nontrivial $\iff$ it intersects (i.e., contains a point from) every tail of $u$. By the definition of a subnet, every subnet of a nontrivial net is nontrivial.}).

A net $u:I\rightarrow X$ is \textbf{eventually in} (resp., \textbf{frequently in}) $A\subset X$, written $T_u\subset A$ (resp., $S_u\subset A$), if $A$ contains a tail (resp., subnet) of $A$. Given a collection of open sets $\mathcal{O}\subset\tau$, the \textbf{derived $\mathcal{O}$-filter} of $u$ is $\mathcal{F}^\mathcal{O}_u:=\{A\subset X:T_u\subset O\subset A$, for some $O\in\mathcal{O}\}$. The \textbf{derived filter} of $u$ is $\mathcal{F}_u:=\mathcal{F}_u^\tau$ (i.e., the derived $\tau$-filter of $u$). We note that for any $A\subset X$, $u$ is either eventually in $A$ or frequently in $X\backslash A$.
\end{dfn}

The notion of \emph{net-approach} (or approach between nets), which we will now define, is an extension of the notion of \emph{topological convergence} that extends the notion of \emph{topological cauchyness} from uniform spaces to base spaces in an induction-free way.

\begin{dfn}[{Approach, Cauchy net, Derived net, Cauchy filter}]\label{ApprchDfn}
Let $u=\{u_i\}_{i\in I}$, $v=\{v_j\}_{j\in J}$ be nets in $X$. We write $\mathcal{O}_\varepsilon(u):=\{O_\varepsilon(u_i)\ni u_i\}_{i\in I}\subset\mathcal{B}_\varepsilon$ for a homogeneous selection of $\mathcal{B}_\varepsilon$-neighborhoods $O_\varepsilon(u_i)$ of points $u_i$ of $u$. We say $u$ \textbf{approaches} $v$ (written  $u\rightarrowtail v$) if for any $\mathcal{O}_\varepsilon(v)\subset\mathcal{B}_\varepsilon$, there exists $j^{\mathcal{O}_\varepsilon(v)}\in J$ such that for each $j\geq  j^{\mathcal{O}_\varepsilon(v)}$, $ O_\varepsilon(v_j)$ contains a tail of $u$. (That is, $u\not\rightarrowtail v$ iff there exists $\mathcal{O}_{\varepsilon_o}(v)\subset\mathcal{B}_{\varepsilon_o}$ such that for any $j\in J$, there exists $j_o(j)\geq j$ such that $O_{\varepsilon_o}(v_{j_o(j)})$ excludes a subnet [i.e., does not contain a tail] of $u$. Equivalently, $u\not\rightarrowtail v$ iff there exists $\mathcal{O}_{\varepsilon_o}(v\circ\psi_o)\subset\mathcal{B}_{\varepsilon_o}$ (for some $\varepsilon_o\in\mathcal{E}$ and some subnet $v\circ\psi_o:J\rightarrow X$) such that each $O_{\varepsilon_o}(v\circ\psi_o(j))$ excludes a subnet of $u$.) We say $u$ is a \textbf{cauchy net} if all subnets of $u$ approach each other, i.e., for all subnets $u\circ\phi$ and $u\circ\psi$, we have $u\circ\phi\rightarrowtail u\circ\psi$. A filter $\mathcal{F}$ in $X$ is a \textbf{cauchy filter} if every \textbf{derived net} $w:(\mathcal{F},\supset)\rightarrow X,~F\mapsto w(F)\in F$ of $\mathcal{F}$ is cauchy ({\bf footnote}\footnote{
Equivalently, $\mathcal{F}$ is \textbf{cauchy} if the unique associated net $w_\mathcal{F}$ (which satisfies $\mathcal{F}=\mathcal{F}_{w_\mathcal{F}}$) of $\mathcal{F}$ is cauchy, where
\[
\textrm{$w_\mathcal{F}:(\{(x,F):x\in F\in\mathcal{F}\},\leq)\rightarrow X$, $(x,F)\mapsto x$, {~} where ``$(x,F)\leq (x',F')$ iff $F\supset F'$''.}
\]
We note that given a net $u$, we have $u\neq w_{\mathcal{F}_u}$ in general. That is, $F\mapsto w_\mathcal{F}$ is injective but $u\mapsto \mathcal{F}_u$ is not.
}).
\end{dfn}

\begin{question}\label{AprDfnQsn}
Can the definition of approach between nets (in Definition \ref{ApprchDfn}) be given in terms of filters?
\end{question}

\begin{dfn}[{Convergent net, Limit point}]\label{CovNetDfn}
A net $u$ \textbf{converges to} a point $x\in X$ (called a \textbf{limit} or \textbf{limit point} of $u$, which is not necessarily unique, and written $u\rightarrow x$) if $u\rightarrowtail\{x\}$ (i.e., every $\mathcal{B}$-neighborhood of $x$ contains a tail of  $u$), where $u\rightarrowtail\{x\}$ is also written as $u\rightarrowtail x$. (That is, $u\not\rightarrow x$ iff there exists a $\mathcal{B}$-neighborhood of $x$ that excludes a subnet of $u$.) We say $u$ \textbf{converges} (or $u$ is \textbf{convergent}) in $X$ if $u$ has a limit point in $X$.
\end{dfn}

\begin{rmk}\label{CmpCntRmk}
Recall that given spaces $X$ and $Y$, (i) a set $A\subset X$ is closed iff every net in $A$ that converges in $X$ has all of its limit points lying in $A$, (ii) $X$ is compact iff every net in $X$ has a convergent subnet, and (iii) a map $f:X\rightarrow Y$ is continuous iff it maps convergent nets to convergent nets.
\end{rmk}

\begin{dfn}[{Completeness, Completion, Continuity, Cauchy continuity, Uniform continuity}]\label{ComplDfn}
Let $X,Y$ be base spaces. $X$ is \textbf{complete} if every cauchy net in $X$ converges in $X$. $Y$ is a \textbf{completion} of $X$ if $Y$ is (i) complete and (ii) contains $X$ as a uniformly embedded subspace (i.e., $Y$ is complete, there is a uniform embedding $h:X\hookrightarrow Y$, and $cl_Y(h(X))=Y$). A map $f:X\rightarrow Y$ is \textbf{continuous} (resp., \textbf{cauchy continuous}) if it maps convergent nets to convergent nets (resp., cauchy nets to cauchy nets), i.e., every convergent net $u\rightarrow x$ (resp, cauchy net $u$) in $X$ gives a convergent net $f\circ u\rightarrow f(x)$ (resp., cauchy net $f\circ u$) in $Y$. The map is \textbf{uniformly continuous} (or \textbf{uniform}) if it maps approaching nets to approaching nets (i.e., for any approach of nets $u\rightarrowtail v$ in $X$, we get an approach of nets $f\circ u\rightarrowtail f\circ v$ in $Y$). ({\bf footnote}\footnote{
See the footnote in Definition \ref{IsoBspsDfn} for an alternative meaning/interpretation of a uniform map (as a morphism in the category of base spaces).
}).
\end{dfn}

Let $X$ be a base space. Given $x,y\in X$, we say $x$ \textbf{approaches} $y$ (written $x\rightarrowtail y$) if $\{x\}\rightarrowtail\{y\}$ (i.e., every $\mathcal{B}$-neighborhood of $y$ contains $x$), and we say $x,y$ \textbf{approach each other} (written $x\sim y$) if $\{x\}\rightarrowtail\{y\}$ and $\{y\}\rightarrowtail\{x\}$ (i.e., $x$ and $y$ have the same $\mathcal{B}$-neighborhoods).
({\bf footnote}\footnote{The operation $\sim$ is an equivalence relation on $X$. With the set of equivalence classes $X_\sim=\{[x]:x\in X\}$, the associated quotient topology on $X_\sim$ is Hausdorff (with the Hausdorff quotient space $X_\sim$ called the \textbf{Hausdorffization} of $X$).
}).

\begin{lmm}[{Net-approach is transitive and hereditary}]\label{NetAppHT}
Consider nets in a base space. (i) Approach of nets is transitive, i.e., if $u\rightarrowtail v\rightarrowtail w$, then $u\rightarrowtail w$.
(ii) Approach of nets is (left) hereditary, i.e., $u\rightarrowtail v$ if and only if every subnet $u\circ\phi\rightarrowtail v$.
\end{lmm}
\begin{proof}
(i) $u\rightarrow v$ means $\forall \{v_j\in O_\varepsilon(v_j)\in\mathcal{B}_\varepsilon\}_{j\in J}$, $\exists j^{\mathcal{O}_\varepsilon(v)}$ ($\mathcal{O}_\varepsilon(v):=\{O_\varepsilon(v_j)\}_{j\in J}$) such that $\forall j\geq  j^{\mathcal{O}_\varepsilon(v)}$, $ O_\varepsilon(v_j)$ contains a tail of  $u$. $v\rightarrow w$ means $\forall \{w_k\in P_\varepsilon(w_k)\in\mathcal{B}_\varepsilon\}_{k\in K}$, $\exists k^{\mathcal{P}_\varepsilon(w)}$ ($\mathcal{P}_\varepsilon(w):=\{P_\varepsilon(w_k)\}_{k\in K}$) such that $\forall k\geq  k^{\mathcal{P}_\varepsilon(w)}$, $ P_\varepsilon(w_k)$ contains a tail of  $v$. So, given $\{P_\varepsilon(w_k)\in\mathcal{B}_\varepsilon\}_{k\in K}$, by choosing $\{O_\varepsilon(v_j)\}_{j\in J}$ such that for each $k\geq  k^{\mathcal{P}_\varepsilon(w)}$, $\mathcal{P}_\varepsilon(w_k)$ contains some $O_\varepsilon(v_j)$ for $j\geq j^{\mathcal{O}_\varepsilon(v)}$ (hence contains a tail of $u$), we see that $u\rightarrowtail w$.

(ii) This is immediate from the definition of approach between nets.
\end{proof}

\begin{dfn}[{$lsb$-space, $csb$-space, $sb$-space}]\label{SymBsSp}
We will call $\mathcal{B}$ a \textbf{locally symmetric base}, making $X=(X,\tau,\mathcal{B})$ a \textbf{locally symmetric base space} or \textbf{$lsb$-space} (resp., \textbf{cauchy symmetric base}, making $X=(X,\tau,\mathcal{B})$ a \textbf{cauchy symmetric base space} or \textbf{$csb$-space}) if the grading $\mathcal{B}=\bigcup_{\varepsilon\in\mathcal{E}}\mathcal{B}_\varepsilon$ is such that if $u\rightarrowtail\{x\}$, then $\{x\}\rightarrowtail u$ (resp., if $v$ is cauchy and $u\rightarrowtail v$, then $v\rightarrowtail u$). Similarly, the base is a \textbf{symmetric base} (making the base space a \textbf{symmetric base space} or \textbf{$sb$-space}) if its grading is such that $u\rightarrowtail v$ implies $v\rightarrowtail u$ (i.e., approach is symmetric). We have the containment chain ``$sb$-spaces $\subset$ $csb$-spaces $\subset$ $lsb$-spaces $\subset$ base spaces''.

We note that $\{x\}\rightarrowtail u$ iff for any $\mathcal{O}_\varepsilon(u)\subset\mathcal{B}_\varepsilon$, there exists $i^{\mathcal{O}_\varepsilon(u)}$ such that $x\in O_\varepsilon(u_i)$ for every $i\geq i^{\mathcal{O}_\varepsilon(u)}$ (that is, $\{x\}\not\rightarrowtail u$ iff there exist neighborhoods $\mathcal{O}_{\varepsilon_o}(u\circ\phi_o)\subset \mathcal{B}_{\varepsilon_o}$ that each exclude $x$, for some $\varepsilon_o\in\mathcal{E}$ and a subnet $u\circ\phi_o$).
({\bf footnote}\footnote{
The standard uniform base space for a uniform space is an example of a $lsb$-space. Indeed: Let a net $w\rightarrow x$, i.e., every $B_U(x)^o\supset w_{[i_{x,U},I]}$ (for some $i_{x,U}\in I$). Fix $V\in\mathcal{U}$, which automatically implies $w_i\in B_V(w_i)^o$ for all $i\in I$. Then with $U:=V^{-1}$, we get $B_{V^{-1}}(x)^o\supset w_{[i_{x,V^{-1}},I]}$, and so $x\in B_V(w_i)^o$ for all $i\geq i_{x,V^{-1}}$ (where the inverse of a binary relation $R\subset X\times X$ is $R^{-1}:=((b,a):(a,b)\in R)$). Thus $w\rightarrowtail x$ implies $\{x\}\rightarrowtail w$. Similarly, $\{x\}\rightarrowtail w$ implies $w\rightarrowtail x$.

It follows similarly that any $u$-space $X$ (via a continuous map $u:X\times X\rightarrow (Z,\mathcal{E})$) with the following properties is a $lsb$-space (hence $lsb$-spaces strictly contain uniform spaces):
\begin{enumerate}
\item If $\varepsilon,\varepsilon'\in\mathcal{E}$, then $\varepsilon\cap\varepsilon'\in\mathcal{E}$.
\item $\forall\varepsilon\in\mathcal{E}$, $\exists\varepsilon'\in\mathcal{E}$ such that $u^{-1}(\varepsilon')\circ u^{-1}(\varepsilon')\subset u^{-1}(\varepsilon)$ and $u(u^{-1}(\varepsilon')\circ u^{-1}(\varepsilon'))\in\mathcal{E}$.
\item $\forall\varepsilon\in\mathcal{E}$, $\exists\varepsilon'\in\mathcal{E}$ such that $u^{-1}(\varepsilon')\subset[u^{-1}(\varepsilon)]^{-1}$.
\end{enumerate}
}).
\end{dfn}

\begin{rmk}[{Equivalent nets}]\label{EqvNetRmk}
In a base space, we say two nets $u$ and $v$ are \textbf{(approach-) equivalent}, written $u\sim v$, if $u\rightarrowtail v$ and $v\rightarrowtail u$. In a base space, the following hold:
\begin{enumerate}
\item  If $u\sim v$, then $u\rightarrow x$ iff $v\rightarrow x$.
\item $u$ is cauchy iff for all subnets $u\circ\phi$ and $u\circ\psi$, we have $u\circ\phi\sim u\circ\psi$ (iff $u\sim u$).
\end{enumerate}
In a $lsb$-space, the following holds:
\begin{enumerate}[resume]
\item $u\rightarrow x$ iff $u\sim\{x\}$.
\end{enumerate}
\end{rmk}

\begin{lmm}
Consider nets in a $lsb$-space. (i) Every subnet of a cauchy net is cauchy. (ii) A cauchy net converges (iff every subnet converges) iff one subnet converges. (This follows from transitivity of approach.)
\end{lmm}

\begin{lmm}
In a $lsb$-space, every convergent net is a cauchy net.
\end{lmm}
\begin{proof}
Let $u$ be convergent in a $lsb$-space $X$, $x\in X$ any limit of $u$, and $u\circ\phi,u\circ\psi$ subnets of $u$. Then $u\sim x$ implies $u\circ\phi\sim x$ and $u\circ\psi\sim x$. Since $u\circ\phi\sim x\sim u\circ\psi$, by transitivity of approach, $u\circ\phi\sim u\circ\psi$.
\end{proof}

\begin{prp}\label{UnifCsbPrp}
(i) A uniform space is a $csb$-space. (ii) Uniform spaces are strictly contained in $lsb$-spaces.
\end{prp}
\begin{proof}
(i) Let $X$ be a uniform space. Suppose $u\rightarrowtail v$. If $u\rightarrowtail x$ (for any $v$), then by the footnote in Definition \ref{SymBsSp} above, $\{x\}\rightarrowtail u$, which implies $\{x\}\rightarrowtail v$ (by transitivity), which implies $v\rightarrowtail x$ (by the footnote), which implies $v\rightarrowtail u$ (by transitivity). Therefore, if $v\rightarrowtail x$ (for any $u$), then $u\rightarrowtail x$ by transitivity, and so $v\rightarrowtail u$ as before. Hence, in a uniform space, approach is symmetric on convergent nets.

It follows that approach is also symmetric on cauchy nets, since every uniform space has a completion, in which every cauchy net converges: Indeed, a net in a uniform space is cauchy iff, via the uniform embedding, it is cauchy and so convergent in the completion. (\textbf{NB:} Thus, in a uniform space, a non-cauchy net can neither approach, nor be approached by, a cauchy net.)

(ii) This is evident from the footnote (in Definition \ref{SymBsSp}) above.
\end{proof}

\begin{thm}\label{CompCsbThm}
(i) A compact $lsb$-space is a $sb$-space. (ii) A continuous map on a compact $sb$-space is uniform. (iii) A continuous map between $lsb$-spaces is uniform on compact subsets.
\end{thm}
\begin{proof}
(i) Let $X$ be a compact $lsb$-space. Since every compact space is uniquely uniformizable, $X$ may be regarded as a compact uniform space. By Proposition \ref{UnifCsbPrp} (and the fact that every net in $X$ has a convergent subnet), if $u\rightarrowtail v$ (hence a convergent subnet of $u$ approaches $v$, making $v$ convergent) then $u$ and $v$ are both convergent, and so $v\rightarrowtail u$.

(ii) In a compact $sb$-space $X$, the only approaching nets are convergent nets (since every net has a convergent subnet), which are mapped by a continuous map $f:X\rightarrow Y$ to approaching convergent nets in the compact image $f(X)$.
\end{proof}

\begin{dfn}[{Precompactness (or Total boundedness)}]\label{PrecompDfn}
A base space $X$ is said to be \textbf{precompact} (or \textbf{totally bounded}) if every net in $X$ has a cauchy subnet.
\end{dfn}

\begin{question}\label{BodyQuest2} 
Consider $lsb$-spaces. It is clear that if a map is uniform, then it is cauchy continuous. When is the converse true, i.e., on which $lsb$-spaces (e.g., precompact spaces in the context of metric spaces) is a cauchy continuous map uniform?
\end{question}

\begin{thm}\label{CpctSpThm}
A $lsb$-space is compact if and only if it is both complete and precompact.
\end{thm}
\begin{proof}
It is clear that a compact $lsb$-space is both complete and precompact. Conversely, let a $lsb$-space $X$ be complete and precompact. Then any net in $X$ has a cauchy subnet (by precomactness) which converges (by completeness), i.e., every net in $X$ has a convergent subnet.
\end{proof}

\begin{dfn}[{Locally complete base space, Locally compact space}]
A base space $X=(X,\tau,\mathcal{B})$ is \textbf{locally complete} if every point of $X$ has a $\mathcal{B}$-neighborhood that is contained in a complete subset. Similarly, a space is \textbf{locally compact} if every point has a neighborhood that is contained in a compact subset.
\end{dfn}

\begin{rmk}[{Some basic results}]\label{BasicResRmk}~
\begin{enumerate}[leftmargin=0.8cm]
\item Every compact $lsb$-space is complete, since every cauchy net in it converges (as a cauchy net with a convergent subnet). Hence every locally compact Hausdorff $lsb$-space is locally complete.
\item An indiscrete $lsb$-space (as well as any finite $lsb$-space) is compact, hence complete.
\item A discrete $lsb$-space is complete (since every non ultimately constant net has subnets that do not approach each other, and so the only cauchy nets are ultimately constant nets), but an infinite discrete $lsb$-space is not compact (since the open cover by singletons has no finite subcover).
\item A complete subspace of a $lsb$-space is closed. Indeed, if a net in the subspace converges in the ambient/containing $lsb$-space, then the net is cauchy in the subspace, and hence converges in the subspace.
\item A closed subspace of a complete $lsb$-space is complete. To see this, let $X$ be a complete $lsb$-space and $A\subset X$ a closed subspace. Then a cauchy net in $A$ converges in $X$ (since $X$ is complete), and so converges in $A$ (since $A$ is closed).
\end{enumerate}
\end{rmk}

\begin{thm}[{Baire's theorem}]\label{BaireThm}
A locally complete regular Hausdorff $lsb$-space $X=(X,\tau,\mathcal{B})$ cannot be written as a countable union of nowhere dense sets. ({\bf footnote}\footnote{
(i) Recall that $X$ is \textbf{regular} (resp., \textbf{normal}) iff for any $x\in X$ (resp., closed set $C\subset X$), every neighborhood $x\in O\subset X$ (resp., $C\subset O\subset X$) contains the closure of another neighborhood $x\in U\subset \overline{U}\subset O\subset X$ (resp., $C\subset U\subset \overline{U}\subset O\subset X$).

(ii) Recall that $A\subset X$ is \textbf{nowhere dense} iff $(\overline{A})^o=\emptyset$, iff $(A^c)^o$ is dense in $X$.
}).
\end{thm}
\begin{proof}
Let $A_1,A_2,\cdots\subset X$ be nowhere dense subsets of $X$. Since $A_1$ is nowhere dense, there is a point $x_1\in X\backslash A_1$ and a complete set $C_1(x_1)\ni x_1$ that does not meet $A_1$. Since $A_2$ is nowhere dense, hence nowhere dense in $C(x_1)$, there is a point $x_2\in C(x_1)\backslash A_2$ and a complete set $C(x_2)\ni x_2$, with $C(x_2)\subset C(x_1)$, that does not meet $A_1\cup A_2$. Continuing this way, we get a nested sequence  $C(x_1)\supset C(x_2)\supset C(x_3)\supset\cdots$ of complete sets such that for each $n$, $C(x_n)\cap\bigcup_{i=1}^nA_i=\emptyset$ (hence $\bigcap_{n=1}^\infty C(x_n)\cap\bigcup_{n=1}^\infty A_n=\emptyset$). Let $C:=\bigcap_nC(x_n)$.

If $C\neq\emptyset$ (so $C$ is a nonempty complete set), then the proof is complete. Suppose $C=\emptyset$.

Consider any two subnets $x_{n_\alpha},x_{n_\beta}$ of $x_n$. Let $\mathcal{O}_\varepsilon(\{x_{n_\beta}\}):=\{O_\varepsilon(x_{n_\beta}): x_{n_\beta}\in O_\varepsilon(x_{n_\beta}),\forall\beta\}\subset\mathcal{B}_\varepsilon$ (for a fixed $\varepsilon$). Since $C(x_n)\rightarrow \emptyset$, there exists $\beta_\varepsilon=\beta^{\mathcal{O}_\varepsilon(\{x_{n_\beta}\})}$ such that for each $\beta\geq n_{\beta_\varepsilon}$, $O_\varepsilon(x_{n_\beta})$ contains every member of a tail of $\{C(x_n)\}$ (hence of a tail of $\{C(x_{n_\alpha})\}$), and so contains a tail of $\{x_{n_\alpha}\}$. Therefore $\{x_n\}$ is cauchy (hence converges in each $C(x_n)$ to a point in $C$, which gives a contradiction).

Therefore $C\subset\left(\bigcup_nA_n\right)^c$ is a nonempty complete set.
\end{proof}

%%########################################################################################################################
\section{\textnormal{\bf Completion of base spaces}}\label{CmplBasSpSec}
%%########################################################################################################################
%%########################################################################################################################
\begin{lmm}[{Marginal convergence preserves net-approach}]\label{BcnvLimit}
Consider nets in a base space. (As in the diagram below) If $u_{\bullet\bullet}\rightarrowtail v_{\bullet\bullet}$, $u_{\bullet\beta}\rightarrowtail u'_\beta$ (for each $\beta$), and $v_{\bullet j}\rightarrowtail v'_j$ (for each $j$), then $u'_\bullet\rightarrowtail v'_\bullet$.
\end{lmm}
\begin{comment}
\bc\bt
B^{k+1}\ar[rrrr,dashed,"\widetilde{F}"] &&&& \mathbb{R}\ar[d,"p(t)=e^{it}"]\\
S^k\ar[u,hook,"i"]\ar[urrrr,dashed,near start,"\widetilde{f}"]\ar[rrrr,near end,"f"] &&&& S^1
\et\ec
\end{comment}
\begin{center}
\begin{tikzcd}
u_{\bullet\bullet}\ar[d,tail] & \forall\beta,~~~u_{\bullet\beta}\ar[rr,tail]  && u'_\beta  & u'_\bullet \ar[d,dashed,tail]\\
v_{\bullet\bullet}            & \forall j,~~~v_{\bullet j}   \ar[rr,tail]  && v'_j & v'_\bullet
\end{tikzcd}
\end{center}
\begin{proof}
Assume $u_{\bullet\bullet}\rightarrowtail v_{\bullet\bullet}$ (i.e., $\forall v_{ij}\in O_\varepsilon(v_{ij})\in\mathcal{B}_\varepsilon$, $\exists$ $(i^{\mathcal{O}_\varepsilon(v_{\bullet\bullet})},j^{\mathcal{O}_\varepsilon(v_{\bullet\bullet})})$ such that $\forall(i,j)\geq (i^{\mathcal{O}_\varepsilon(v_{\bullet\bullet})},j^{\mathcal{O}_\varepsilon(v_{\bullet\bullet})})$, $O_\varepsilon(v_{ij})$ contains a tail of $u_{\bullet\bullet}$, where $\mathcal{O}_\varepsilon(v_{\bullet\bullet}):=\{O_\varepsilon(v_{ij}):(i,j)\in I\times J\}$
),

$u_{\bullet\beta}\rightarrowtail u'_\beta$ for each $\beta$ (i.e., every neighborhood $u'_{\beta}\in O_\varepsilon(u'_\beta)\in\mathcal{B}_\varepsilon$ contains a tail of  $u_{\bullet\beta}$), and

$v_{\bullet j}\rightarrowtail v'_j$ for each $j$ (i.e., every neighborhood $v'_j\in O_\varepsilon(v'_j)\in\mathcal{B}_\varepsilon$ contains a tail of  $v_{\bullet j}$).

Suppose $u'_\bullet\not\rightarrowtail v'_\bullet$. That is, there exists $\mathcal{O}_{\varepsilon_o}(v'_\bullet\circ\psi_o)\subset\mathcal{B}_{\varepsilon_o}$ (for some $\varepsilon_o\in\mathcal{E}$ and some subnet $v'_\bullet\circ\psi_o:J\rightarrow X$) such that each $O_{\varepsilon_o}(v'_\bullet\circ\psi_o(j))$ excludes a subnet of $u'_\bullet$.

Fix $j\in J$ such that $\psi_o(j)\geq j^{\mathcal{O}_{\varepsilon_o}(v_{\bullet\bullet})}$. Then the convergences $\{v_{\bullet j}\rightarrowtail v'_j\}_{\forall j\in J}$ imply $O_{\varepsilon_o}(v'_\bullet\circ\psi_o(j))$ contains a tail of $v_{\bullet \psi_o(j)}$, and so by $u_{\bullet\bullet}\rightarrowtail v_{\bullet\bullet}$, contains a tail of $u_{\bullet\bullet}$.  Since a subnet of $u'_\bullet$ lies outside of $O_{\varepsilon_o}(v'_\bullet\circ\psi_o(j))$, by the convergences $\{u_{\bullet\beta}\rightarrowtail u'_\beta\}_{\forall\beta}$, a subnet of $u_{\bullet\bullet}$ also lies outside of $O_{\varepsilon_o}(v'_\bullet\circ\psi_o(j))$ (a contradiction).
\end{proof}

\begin{thm}[{A base space is a net-approach space}]\label{NtAppSpThm}
In any base space, approach of nets (Definition \ref{ApprchDfn}) satisfies Axioms (\ref{NetAppAx1})-(\ref{NetAppAx4}) of Definition \ref{NtApSumrDfn}.
\end{thm}
\begin{proof}
Axiom (\ref{NetAppAx1}) is immediate by construction. Axioms (\ref{NetAppAx2})-(\ref{NetAppAx4}) follow from Lemmas \ref{NetAppHT} and \ref{BcnvLimit}.
\end{proof}

\begin{lmm}[{Tail exclusion lemma}]\label{TailExcLmm}
Let $X=(X,\tau,\mathcal{B})$ be a base space, $u:I\rightarrow X$ a cauchy net, and $x\in X$. Then $u\not\rightarrowtail\{x\}$ iff $x$ has a $\mathcal{B}$-neighborhood that excludes a tail of $u$.
\end{lmm}
\begin{proof}
($\Rightarrow$): If $u\not\rightarrowtail \{x\}$, then a $\mathcal{B}$-neighborhood $O\ni x$ excludes a subnet $u\circ\phi$ of $u$. Suppose $O$ contains a subnet $u\circ\psi$ of $u$. Since $u\circ\phi\rightarrowtail u\circ\psi$, for any $u_{\psi_\alpha}\in O_\varepsilon(u_{\psi_\alpha})\in\mathcal{B}_\varepsilon$ with $O_\varepsilon(u_{\psi_\alpha})\subset O$, there exists $\alpha^{\mathcal{O}_\varepsilon(u\circ\psi)}$ (where $\mathcal{O}_\varepsilon(u\circ\psi):=\{O_\varepsilon(u_{\psi_\alpha}):\forall\alpha\}$) such that for any $\alpha\geq \alpha^{\mathcal{O}_\varepsilon(u\circ\psi)}$, $O_\varepsilon(u_{\psi_\alpha})\subset O$ contains a tail of $u\circ\phi$ (a contradiction). ($\Leftarrow$): The converse is immediate.
\end{proof}

\begin{crl}\label{TailExcCrl}
Let $X=(X,\tau,\mathcal{B})$ be a base space and $u:I\rightarrow X$ a cauchy net. A $\mathcal{B}$-open set $O$ excludes a subnet of $u$ iff it excludes a tail of $u$.
\end{crl}

\begin{crl}\label{TailExcCrlII}
Let $X=(X,\tau,\mathcal{B})$ be a base space, $u:I\rightarrow X$ a cauchy net, and $v:J\rightarrow X$ a net. Then $u\not\rightarrowtail v$ iff there exist $\varepsilon_o$ and $\mathcal{O}_{\varepsilon_o}(v)\subset\mathcal{B}_{\varepsilon_o}$ such that for any $j\in J$, there exists $j_o(j)\geq j$ such that $O_{\varepsilon_o}(v_{j_o(j)})$ excludes a tail of $u$. Equivalently, $u\not\rightarrowtail v$ iff there exists $\mathcal{O}_{\varepsilon_o}(v\circ\psi_o)\subset\mathcal{B}_{\varepsilon_o}$ (for some $\varepsilon_o\in\mathcal{E}$ and some subnet $v\circ\psi_o:J\rightarrow X$) such that each $O_{\varepsilon_o}(v\circ\psi_o(j))$ excludes a tail of $u$.
\end{crl}

\begin{lmm}\label{NFBscsLmm}
Let $X$ be a base space and $u,v$ nets in $X$. (i) If $u$ is a derived net of $\mathcal{F}_v$, then $\mathcal{F}_u\supset\mathcal{F}_v$ (so, $\mathcal{F}_u=\mathcal{F}_{v\circ\phi}$ for a subnet $v\circ\phi$). (ii) $u\rightarrowtail v$ $\Rightarrow$ $\mathcal{F}_u\supset\mathcal{F}_v$ (so, $u\sim v$ $\Rightarrow$ $\mathcal{F}_u=\mathcal{F}_v$). (iii) If $v$ is cauchy (i.e., $v\sim v$), then $u\rightarrowtail v$ $\iff$ $\mathcal{F}_u\supset\mathcal{F}_v$ (so, if $u,v$ are cauchy, then $u\sim v$ $\iff$ $\mathcal{F}_u=\mathcal{F}_v$).
\end{lmm}
\begin{proof}
This follows directly from the definitions and \cite[Theorem 4.6]{Freiwald2014}.
\end{proof}

\begin{lmm}\label{CchyFltLmm}
Let $X$ be a $csb$-space and $u$ a cauchy net in $X$. Then $\mathcal{F}_u$ is a cauchy filter in $X$.
\end{lmm}
\begin{proof}
Let $v:(\mathcal{F}_u,\supset)\rightarrow X,~A\mapsto v(A)\in A$ be a derived net of $\mathcal{F}_u$. Since $\mathcal{F}_v\supset\mathcal{F}_u$, we have $v\rightarrowtail u$ due to $u\rightarrowtail u$ by cauchyness of $u$. Therefore $v$ is cauchy by cauchy symmetry (since $v\rightarrowtail u$ $\Rightarrow$ $u\rightarrowtail v$).
\end{proof}

\begin{lmm}\label{tUConLmm4}
Let $X$ be a Hausdorff $lsb$-space, $E\subset X$ a dense subset, and $x_\bullet\rightarrowtail\widetilde{x}_\bullet$ approaching nets in $X$. Then there exist approaching bi-nets $e_{\bullet\bullet}\rightarrowtail \widetilde{e}_{\bullet\bullet}$ in $E$ such that $e_{\bullet\alpha}\rightarrow x_\alpha$ $\forall\alpha$ and $\widetilde{e}_{\bullet\widetilde{\alpha}}\rightarrow \widetilde{x}_{\widetilde{\alpha}}$ $\forall\widetilde{\alpha}$.
\end{lmm}
\begin{proof}
Since $E$ is dense in $X$, for each $\alpha$ (resp., $\widetilde{\alpha}$), we can pick a convergent net $E\ni e_\bullet(x_\alpha)\rightarrowtail x_\alpha$ (resp., $E\ni \widetilde{e}_\bullet(\widetilde{x}_{\widetilde{\alpha}})\rightarrowtail \widetilde{x}_{\widetilde{\alpha}}$) with each $e_\beta(x_\bullet)$ (resp., $\widetilde{e}_{\widetilde{\beta}}(\widetilde{x}_\bullet)$) chosen to be a derived net of the filter $\mathcal{F}_{x_\bullet}$ (resp., $\mathcal{F}_{\widetilde{x}_\bullet}$). Again, using the density of $E$, we can choose these nets such that the following hold: (i) Each $e_\beta(x_\bullet)\rightarrowtail x_\bullet$ (resp.,
each $\widetilde{e}_{\widetilde{\beta}}(\widetilde{x}_\bullet)\rightarrowtail \widetilde{x}_\bullet$) and (ii) $e_\bullet(x_\bullet) \rightarrowtail \widetilde{e}_\bullet(\widetilde{x}_\bullet)$.
\end{proof}

\begin{thm}[{Uniform extension theorem}]\label{tUConThm}
Let $X$,$Y$ be Hausdorff $lsb$-spaces and $f : E\subset X\rightarrow Y$ a map. If (i) $f$ is uniform, (ii) $E$ is dense, and (iii) $Y$ is complete, then $f$ extends to a unique uniform map $F : X\rightarrow Y$.
\end{thm}
\begin{proof}
Consider the map $F:X\rightarrow Y,~x\mapsto\lim f\circ u$ for any net $u\in E$ such that $u\rightarrow x$ (which is well-defined due to the Hausdorff property). We need to show $F$ is uniform (hence unique, since continuous maps that agree on a dense set agree everywhere). Let $x_\bullet\rightarrowtail \widetilde{x}_\bullet$ in $X$. Pick convergent nets $E\ni e_\bullet({x_\alpha})\rightarrow x_\alpha$ $\forall\alpha$ and $E\ni \widetilde{e}_\bullet(\widetilde{x}_{\widetilde{\alpha}})\rightarrow \widetilde{x}_{\widetilde{\alpha}}$ $\forall\widetilde{\alpha}$ such that $e_\bullet({x_\bullet})\rightarrowtail \widetilde{e}_\bullet(\widetilde{x}_\bullet)$ (which is possible by Lemma \ref{tUConLmm4}). Then $F(e_\bullet\big({x_\alpha})\big)=f(e_\bullet\big({x_\alpha})\big)\rightarrow F(x_\alpha)$ $\forall\alpha$ and $F(\widetilde{e}_\bullet\big(\widetilde{x}_{\widetilde{\alpha}})\big)=f(\widetilde{e}_\bullet\big(\widetilde{x}_{\widetilde{\alpha}})\big)\rightarrow F(\widetilde{x}_{\widetilde{\alpha}})$ $\forall\widetilde{\alpha}$. By uniform continuity of $f$, $y_{\bullet\bullet}:=f(e_\bullet\big({x_\bullet}))$ and $\widetilde{y}_{\bullet\bullet}:=f(\widetilde{e}_\bullet\big(\widetilde{x}_\bullet))$ are approaching bi-nets (with $y_{\bullet\bullet}\rightarrowtail \widetilde{y}_{\bullet\bullet}$) satisfying $y_{\bullet\alpha}\rightarrow y'_\alpha:=F(x_\alpha)$ $\forall\alpha$ and $\widetilde{y}_{\bullet\widetilde{\alpha}}\rightarrow \widetilde{y}'_{\widetilde{\alpha}}:=F(\widetilde{x}_{\widetilde{\alpha}})$ $\forall\widetilde{\alpha}$. So, by Lemma \ref{BcnvLimit}, $y'_\bullet\rightarrowtail\widetilde{y}'_\bullet$ (hence $F$ is uniform).
\end{proof}

\begin{lmm}\label{tUConLmm3}
Let $X$ be a Hausdorff $lsb$-space, $E\subset X$ a dense subset, and $x_\bullet$ a cauchy net in $X$. Then there exists a cauchy bi-net $e_{\bullet\bullet}$ in $E$ such that $e_{\bullet\alpha}\rightarrow x_\alpha$ for each $\alpha$.
\end{lmm}
\begin{proof}
This uses the same reasoning as in the proof of Lemma \ref{tUConLmm4}.
\end{proof}

\begin{dfn}[{Isomorphism of base spaces}]\label{IsoBspsDfn}
A map $f:(X,\tau_X,\mathcal{B}_X)\rightarrow (Y,\tau_Y,\mathcal{B}_Y)$ is an \textbf{isomorphism of base spaces} (or \textbf{base space isomorphism}) if it is a \textbf{bi-uniform homeomorphism} (i.e., $f$ is a homeomorphism such that both $f$ and $f^{-1}$ are uniform). ({\bf footnote}\footnote{
In the categorical context, a \textbf{morphism of base spaces} is a map $f:(X,\tau_X,\mathcal{B}_X)\rightarrow (Y,\tau_Y,\mathcal{B}_Y)$ between base spaces such that for each $\varepsilon\in\mathcal{E}_Y$, there exists $\varepsilon'\in\mathcal{E}_X$ such that $f^{-1}(\mathcal{B}_{Y,\varepsilon})\subset f^{-1}(\mathcal{B}_{Y,\varepsilon'})$. This is precisely the definition of a \textbf{uniform map} (i.e., uniform maps are morphisms in the category base spaces). Thus, an \textbf{isomorphism of base spaces} is a bijective biuniform morphism of base spaces.
}).
\end{dfn}

\begin{thm}[{Completion of $lsb$-spaces}]\label{BSpCmplThm}
Let $X=(X,\tau,\mathcal{B})$ be a Hausdorff $lsb$-space. There exists a complete Hausdorff $lsb$-space $\widetilde{X}$ and a (bi)uniform embedding $f : X\hookrightarrow\widetilde{X}$ such that $f(X)$ is dense in $\widetilde{X}$. Moreover, $\widetilde{X}$ is unique up to isomorphism (of base spaces).

(NB: By Theorem \ref{tUConThm}, every uniform map $g:X\rightarrow Z$, to a complete $lsb$-space $Z$, factors as $g=\widetilde{g}\circ f:X\stackrel{f}{\hookrightarrow}\widetilde{X}\stackrel{\widetilde{g}}{\longrightarrow}Z$ via a unique uniform map $\widetilde{g}:\widetilde{X}\rightarrow Z$, namely, the unique uniform extension of $g\circ f^{-1}:f(X)\rightarrow Z,~y\mapsto g(f^{-1}(y))$.)
\end{thm}
\begin{proof}
\textbf{Existence:} Let $N(X)$ be a sufficient set of nets in $X$ and $CN(X)$ consist of cauchy members of $N(X)$. For $u,v\in CN(X)$, let $u\sim v$ if $u\rightarrowtail v$ and $v\rightarrowtail u$ (as in Remark \ref{EqvNetRmk}), define $\widetilde{X}:={CN(X)\over\sim}=\{[u]:u\in CN(X)\}$, consider the map $f:X\hookrightarrow C,~x\mapsto [u^x]$ (where $u^x$ is the constant $x$-valued net $\{x\}$), and give $\widetilde{X}$ the topology $\widetilde{\tau}$ generated (as subbase) by the following sets:
\begin{enumerate}[leftmargin=0.7cm]
\item\label{CmplSet1} The sets $\{f(O):\textrm{for}~O\in\tau\}$.
\item\label{CmplSet2} The sets $\{[f(Q)]:Q\in\tau$ and $[f(Q)]$ consisting of all $[u]\in \widetilde{X}$ such that $Q$ is a neighborhood of a tail of $u$ in $X\}$ (i.e., $[f(Q)]:=\{[u]\in\widetilde{X}:Q\in\mathcal{F}_u\cap\tau~\textrm{in $X$}\}\supset f(Q)$).  (Together with (\ref{CmplSet1}) above, this makes every cauchy net in $f(X)$ convergent in $(\widetilde{X},\widetilde{\tau})$.)
\end{enumerate}
Similarly, give $\widetilde{X}=(\widetilde{X},\widetilde{\tau})$ the base $\widetilde{\mathcal{B}}$ generated (as subbase) by the following sets:
\begin{enumerate}[resume,leftmargin=0.7cm]
\item\label{CmplbSet1} The sets $\{f(O):\textrm{for}~O\in\mathcal{B}\}$.
\item\label{CmplbSet2} The sets $\{[f(Q)]_\mathcal{B}:Q\in\mathcal{B}$ open and $[f(Q)]_\mathcal{B}$ consisting of all $[u]\in \widetilde{X}$ such that $Q$ is a $\mathcal{B}$-neighborhood of a tail of $u$ in $X\}$ (i.e., $[f(Q)]_\mathcal{B}:=\{[u]\in\widetilde{X}:Q\in\mathcal{F}_u\cap\mathcal{B}~\textrm{in $X$}\}\supset f(Q)$).
\end{enumerate}

By construction, $f$ is a (bi)uniform embedding and $f(X)$ is dense in $\widetilde{X}$. Also, recall that a net in $\widetilde{X}$ is either frequently in $X$ or eventually in $\widetilde{X}\backslash X$. So, it remains only to ensure that we also have convergence of those cauchy nets with images in $\widetilde{X}\backslash X$.

Let $\widetilde{u}_\bullet$ be a cauchy net in $\widetilde{X}$. Since $f(X)$ is dense in $\widetilde{X}$ by (\ref{CmplSet1})$\&$(\ref{CmplSet2}), pick a cauchy bi-net $\widetilde{v}_{\bullet\bullet}=[v_{\bullet\bullet}]$ in $f(X)$ such that $\widetilde{v}_{\alpha\bullet}\rightarrowtail\widetilde{u}_\alpha$ for each $\alpha$ (which exists by Lemma \ref{tUConLmm3}). Again by (\ref{CmplSet1})$\&$(\ref{CmplSet2}), $\widetilde{v}_{\bullet\bullet}$ converges in $\widetilde{X}$ (as a cauchy net in $f(X)$). Hence, by Lemma \ref{BcnvLimit}, $\widetilde{u}_\bullet$ converges as well. This proves completeness of $(\widetilde{X},\widetilde{\tau},\widetilde{\mathcal{B}})$.

\textbf{Uniqueness:} Since $X$ is Hausdorff, $\widetilde{X}$ is also Hausdorff by uniqueness of the limit of a net in $\widetilde{X}$ (because the limit of a net in $f(X)$, hence also of a net in $\widetilde{X}$, is unique by the construction of $\widetilde{X}$ through quotienting). Let $f_i : X\rightarrow\widetilde{X}_i$, $i\in\{1,2\}$, be (bi)uniform embeddings such that each $\widetilde{X}_i$ is complete and each $f_i(X)$ is dense in $\widetilde{X}_i$. Then we get the injective map
\[
h : f_1(X)\subset\widetilde{X}_1\rightarrow\widetilde{X}_2,~~f_1(x)\mapsto f_2(x)
\]
which is (bi)uniform on $f_1(X)$, since for each $i$, $x_\bullet\rightarrowtail\widetilde{x}_\bullet$ in $X$ iff $f_i(x_\bullet)\rightarrowtail f_i(\widetilde{x}_\bullet)$ (hence $f_1(x_\bullet)\rightarrowtail f_1(\widetilde{x}_\bullet)$ iff $f_2(x_\bullet)\rightarrowtail f_2(\widetilde{x}_\bullet)$).

Since $f_1(X)$ is dense in $\widetilde{X}_1$, and $\widetilde{X}_2$ is complete, it follows by Theorem \ref{tUConThm} that $h$ extends to a unique (bi)uniform injective map $\widetilde{X}_1\rightarrow \widetilde{X}_2$. Hence, by symmetry (and Cantor-Shr\"oder-Bernstein theorem), we get an isomorphism $h : \widetilde{X}_1\rightarrow \widetilde{X}_2$.
\end{proof}

\begin{crl}
A $lsb$-space is precompact if and only if it has a compact completion.
\end{crl}

\begin{dfn}[{Cauchy space, Cauchy filter}]\label{ConvChSpDfn}
Let $X$ be a set and $\mathcal{F}[X]$ the set of filters in $X$. A collection of filters $\mathcal{C}(X)\subset\mathcal{F}[X]$ is a \textbf{cauchy structure} on $X$ (making $X=(X,\mathcal{C}(X))$ a \textbf{cauchy space} and members of $\mathcal{C}(X)$ \textbf{cauchy filters}) if, for any $x\in X$ and $\mathcal{F},\mathcal{G}\in\mathcal{F}[X]$;
\begin{enumerate}
\item $\mathcal{F}_x:=\{A\subset X:x\in A\}\in\mathcal{C}(X)$.
\item If $\mathcal{F}\in\mathcal{C}(X)$ and $\mathcal{G}\supset\mathcal{F}$, then $\mathcal{G}\in\mathcal{C}(X)$.
\item If $\mathcal{F},\mathcal{G}\in\mathcal{C}(X)$ and $F\cap G\neq\emptyset$ for all $F\in\mathcal{F}$ and $G\in\mathcal{G}$, then $\mathcal{F}\cap\mathcal{G}\in\mathcal{C}(X)$.
\end{enumerate}
\end{dfn}

\begin{thm}[{Every $csb$-space is a cauchy space}]\label{EBPCSthm}
Every $csb$-space $X=(X,\tau,\mathcal{B})$ has a cauchy structure that induces its topology (making it a cauchy space).
\end{thm}
\begin{proof}
We simply need to show that our cauchy filters from Definition \ref{ApprchDfn} satisfy the axioms of conventional cauchy filters in Definition \ref{ConvChSpDfn}. Let $\mathcal{F}^\mathcal{B}[X]$ denote the set of all $\mathcal{B}$-filters in $X$. Fix any $\mathcal{F},\mathcal{G}\in \mathcal{F}^\mathcal{B}[X]$.

(i) If $\mathcal{F}=\mathcal{F}_x$ for some $x\in X$, then $\mathcal{F}$ is cauchy (because every derived net of $\mathcal{F}$ converges to $x$ and every convergent net in $X$ is cauchy).

(ii) If $\mathcal{F}$ is cauchy and $\mathcal{G}\supset\mathcal{F}$, then $\mathcal{G}$ is cauchy. This is because $\mathcal{F}=\mathcal{F}^\mathcal{B}_u$ for a derived net $u$ of $\mathcal{F}$, and so $\mathcal{G}=\mathcal{F}^\mathcal{B}_{u\circ \phi}$ for some subnet $u\circ \phi$ of $u$ (which is therefore cauchy since $u$ is cauchy). Hence every derived net of $\mathcal{G}$ is cauchy (as a net that apporaches the cauchy net $u\circ\phi$). (\label{NetFiltFN}{\bf footnote}\footnote{
\textbf{Some basic facts about nets and filters (also see Lemma \ref{NFBscsLmm}):} (i) If $u:(\mathcal{F}_v,\supset)\rightarrow X,~A\mapsto u(A)\in A$ is a derived net of a derived filter $\mathcal{F}_v$, then $\mathcal{F}_u\supset\mathcal{F}_v$.  (ii) Every filter $\mathcal{F}$ can be expressed as a derived filter $\mathcal{F}=\mathcal{F}_u$. (iii) $\mathcal{F}_{u\circ\phi}\supset\mathcal{F}_u$ for each subnet $u\circ \phi$. (iv) If $\mathcal{G}\supset \mathcal{F}_u$, then $\mathcal{G}=\mathcal{F}_{u\circ\phi}$ for some subnet $u\circ\phi$.
}).

(iii) If $\mathcal{F},\mathcal{G}$ are both cauchy and $F\cap G\neq\emptyset$ for all $F\in\mathcal{F},G\in\mathcal{G}$, then $\mathcal{F}\cap \mathcal{G}$ is cauchy. This is because a derived net $w$ of $\mathcal{F}\cap\mathcal{G}$ extends to a derived net of $\mathcal{F}$ or $\mathcal{G}$ (which is therefore cauchy), making $w$ a cauchy net (as a tail of a cauchy net).

Therefore, the set of cauchy $\mathcal{B}$-filters $\mathcal{C}\mathcal{F}^\mathcal{B}[X]\subset\mathcal{F}^\mathcal{B}[X]$ in $X$ makes $(X,\mathcal{C}(X))=(X,\mathcal{C}\mathcal{F}^\mathcal{B}[X])$ a cauchy space.
\end{proof}

%%########################################################################################################################
\section{\textnormal{\bf The case of product and function base spaces}}\label{CmPrdFxnSpSec}
%%########################################################################################################################
%%########################################################################################################################
As mentioned just before the start of Section \ref{ApprSec}, in this section in particular, the phrases ``cauchy'', ``complete'', ``uniform'', ``precompact'' should be prefixed by ``\emph{induction-free}'', but for notational convenience, we will drop the prefix. For us, $X^Y$ (resp., $C(Y,X)$) denotes the set of all maps (resp., continuous maps) from $Y$ to $X$ but this is not standard; for example, in \cite{engelking1989}, $C(Y,X)$ is denoted by $X^Y$. Although we have not fully used the traditional language of uniform spaces (which we have, for convenience, simply treated as examples of $u$-spaces), the conceptual contents of \cite[Ch. X]{BbkGT1966} appear to be similar to our treatment of function spaces here (except that we have no preassigned topology inducing structure such as the uniform structure).

\begin{dfn}[{Topologies of uniform, pointwise, compact, and complete convergence}]\label{TpUnfCnvDfn}
Let $Y$ be a set, $X=(X,\tau_X,\mathcal{B}_X)$ a base space, and $\mathcal{F}\subset X^Y$. For $f_\alpha,f\in \mathcal{F}$ and $Z\subset Y$, we say $f_\alpha$ \textbf{converges uniformly to $f$ on $Z$}, written $f_\alpha\stackrel{u|_Z}{\longrightarrow}f$, if for any homogeneous family $\mathcal{O}_\varepsilon=\{O_\varepsilon(f(y))\ni f(y):y\in Z\}\subset\mathcal{B}_{X,\varepsilon}$ of open subsets of $X$, there exists $\alpha^{\mathcal{O}_\varepsilon}$ such that for each $y\in Z$, $f_{[\alpha^{\mathcal{O}_\varepsilon},~]}(y):=\{f_\alpha(y)\}_{\alpha\geq \alpha^{\mathcal{O}_\varepsilon}}\subset O_\varepsilon(f(y))$. If $Z=Y$, we simply say $f_\alpha$ \textbf{converges uniformly} to $f$, written $f_\alpha\stackrel{u}{\longrightarrow}f$.

Given a family of sets $\mathcal{Z}\subset\mathcal{P}(Y)$, a net $\{f_\alpha\}_{\alpha\in\mathcal{A}}\subset\mathcal{F}$ \textbf{converges uniformly with respect to $\mathcal{Z}$} to $f\in \mathcal{F}$ (making it a \textbf{$\mathcal{Z}$-uniformly convergent net}), written $f_\alpha\stackrel{u|_\mathcal{Z}}{\longrightarrow}f$, if $f_\alpha\stackrel{u|_Z}{\longrightarrow}f$ for all $Z\in\mathcal{Z}$.

A subset $\mathcal{C}\subset \mathcal{F}$ is \textbf{$\mathcal{Z}$-uniformly closed} if every $\mathcal{Z}$-uniformly convergent net in $\mathcal{C}$ converges in $\mathcal{F}$ to a point in $\mathcal{C}$. The \textbf{topology of $\mathcal{Z}$-uniform convergence} $\tau_{uc|_\mathcal{Z}}$ on $\mathcal{F}$ is the topology  whose closed sets are the $\mathcal{Z}$-uniformly closed subsets of $\mathcal{F}$. A subbase $\mathcal{S}\mathcal{B}_{uc|_{\mathcal{Z}}}$ for $\tau_{uc|_\mathcal{Z}}$ consists of the sets
{\small\[
\textstyle [Z,O]_{uc|_\mathcal{Z}}:=\{f\in \mathcal{F}:f(y)\in O_y~\forall y\in Z\},~{~}~\textrm{for}~{~}~Z\in\mathcal{Z},~O=\{O_y\}_{y\in Y}\in\prod_{y\in Y}\tau_X=\tau_X{}^Y.
\]}

We have the following special cases of $\tau_{uc|_\mathcal{Z}}$ ({\bf footnote}\footnote{
When working with nets, a convenient way to compare topologies is this: $\tau_1\supset\tau_2$ iff every net that converges in $\tau_1$ also convergence in $\tau_2$ (i.e., convergence in $\tau_1$ implies convergence in $\tau_2$).
}):
\begin{enumerate}[leftmargin=0.7cm]
\item \textbf{Topology of uniform convergence ($\tau_{uc}$)}: $\tau_{uc}:=\tau_{uc}|_{\{Y\}}=\tau_{uc}|_{\mathcal{P}(Y)}$, i.e., $\mathcal{Z}=\{Y\}$ (or equivalently, $\mathcal{Z}=\mathcal{P}(Y)$). ({\bf footnote}\footnote{
   We have $f_\alpha\stackrel{u}{\longrightarrow}f$ $\iff$ $f_\alpha\stackrel{u|_{\{Y\}}}{\longrightarrow}f$, $\iff$ $f_\alpha\stackrel{u|_{\mathcal{P}(Y)}}{\longrightarrow}f$.
    })
\item \textbf{Topology of pointwise convergence ($\tau_p$)}: $\tau_p:=\tau_{uc|_{Singlt(Y)}}$, i.e., $\mathcal{Z}=Singlt(Y):=\big\{\{y\}:y\in Y\big\}$.
\item \textbf{Topology of compact convergence ($\tau_{cc})$}: If $Y$ is a space, $\tau_{\textrm{cc}}:=\tau_{uc|_{K(Y)}}$, i.e., $\mathcal{Z}=K(Y)$ are the compact subsets of $Y$.
\item \textbf{Topology of complete convergence ($\tau_{c_pc}$)}: If $Y$ is a $lsb$-space (or base space if need be) and $\textrm{Cpl}(Y):=\{$complete subsets of $Y\}$, $\tau_{c_pc}:=\tau_{uc|_{\textrm{Cpl}(Y)}}$, i.e., $\mathcal{Z}=\textrm{Cpl}(Y)$.
\end{enumerate}

Given a collection $\mathcal{Z}\subset\mathcal{P}(Y)$, the \textbf{$\mathcal{Z}$-open topology} $\tau_{\mathcal{Z} o}\subset\tau_{uc|_{\mathcal{Z}}}$ on $\mathcal{F}\subset X^Y$ (for $lsb$-spaces $X,Y$) is the topology with a subbase consisting of sets of the form
    \[
    [Z,O]_{co}:=\{f\in \mathcal{F}:f(Z)\subset O\},~~\textrm{for $Z\in \mathcal{Z}$ and $O\in\tau_X$}.
    \]
The topology $\tau_{co}:=\tau_{K(Y)o}$ is known as the \textbf{compact-open topology}. In the same manner, we will call $\tau_{c_po}:=\tau_{Cpl(Y)o}$ the \textbf{complete-open topology}.
\end{dfn}

\begin{rmk}\label{BallTopRmk}
\begin{enumerate}[leftmargin=0.7cm]
\item When $X=(X,u)$ is a $u$-space (with $u:X\times X\rightarrow (Z,\mathcal{E})$ continuous), $\tau_{uc|_{\mathcal{Z}}}$ is precisely the topology on $\mathcal{F}\subset X^Y$ (call it the \textbf{topology of ball-like convergence}) with a subbase of ball-like sets
\[
\mathcal{S}\mathcal{B}_{\textrm{ball}}:=\left\{B^{u,Z}_\varepsilon(f)=B^u_{[Z,\varepsilon]_{uc|_\mathcal{Z}}}(f)~|~\varepsilon\in\mathcal{E},~f\in\mathcal{F},~Z\in \mathcal{Z}\right\},
\]
where $~B^{u,Z}_\varepsilon(f):=\{g\in\mathcal{F}:g(y)\in B^u_\varepsilon(f(y))~\forall y\in Z\}=\big\{g\in\mathcal{F}:\{u(f(y),g(y))\}_{y\in Z}\subset\varepsilon\big\}=\bigcap_{y\in Z}[y,B^u_\varepsilon(f(y))]_p\cap\mathcal{F}$ (see relevant notation for \emph{product base spaces} in Definition \ref{AssocBsp} and \emph{product $u$-spaces} in Definition \ref{PrdUspDfn}).

\item In Definition \ref{TpUnfCnvDfn}, ``\emph{convergence of a net}'' can be accordingly upgraded to ``\emph{approach between nets}'' whenever a topology $\tau$ on $X^Y$ (such as $\tau=\tau_{uc|_{\mathcal{Z}}}$ or another) and a graded base $\mathcal{B}=\bigcup_{\varepsilon\in\mathcal{E}}\mathcal{B}_{\varepsilon}$ for $\tau$ is provided (as in Definitions \ref{AssocBsp} and \ref{PrdUspDfn}), making $X^Y=(X^Y,\tau,\mathcal{B})$ a base space.
\end{enumerate}
\end{rmk}

\begin{question}\label{BodyQuest3} 
It is obvious that the ``singleton-open topology'' equals the ``topology of pointwise convergence'' (i.e., $\tau_{\textrm{Singlt(Y)}o}=\tau_p$).  Also, it is known (\cite[Theorem 8.2.6]{engelking1989}, \cite[Theorem 46.8, for metric spaces]{Munkres2000}) that for uniform spaces $X,Y$, we have $\tau_{co}=\tau_{cc}$ (i.e., $\tau_{K(Y)o}=\tau_{uc|_{K(Y)}}$) on continuous maps $\mathcal{F}=C(Y,X)\subset X^Y$ (which is mainly because continuous maps between uniform spaces are uniform on compact subsets, a fact also true for $lsb$-spaces by Theorem \ref{CompCsbThm}). For which $lsb$-spaces (especially $u$-spaces) $X,Y$ does this remain true?

More generally, given $\mathcal{F}\subset X^Y$ (e.g., $\mathcal{F}$ = $X^Y$, $C(Y,X)$, $UC(Y,X)$, ...), for which $lsb$-spaces $X,Y$ and families of sets $\mathcal{Z}\subset\mathcal{P}(Y)$ do we have ~$\tau_{\mathcal{Z}o}=\tau_{uc|_{\mathcal{Z}}}$? (NB: Since we always have $\tau_{\mathcal{Z}o}\subset\tau_{uc|_{\mathcal{Z}}}$, we only need to find out when $\tau_{\mathcal{Z}o}\supset\tau_{uc|_{\mathcal{Z}}}$.)
\end{question}

\begin{dfn}[{Product space, Product base space, Pointwise cauchyness}]\label{AssocBsp}
Recall that given sets $X$, $Y$, and a family of sets $\{X_y:y\in Y\}$, we write $\prod_{y\in Y}X_y:=\{\textrm{selection maps}~f:Y\rightarrow\bigcup_{y\in Y}X_y,~y\mapsto f(y)\in X_y\}$ and $X^Y:=\prod_{y\in Y}X$. If the $X_y$'s are spaces $X_y=(X_y,\tau_y)$, then $\prod_yX_y$ is often considered a space with respect to the \textbf{product topology} (making it a \textbf{product space}) $\tau_p$ with subbase consisting of sets of the form
\[
\textstyle [y,O]_p:=\{f\in\prod_yX_y:f(y)\in O_y\},~{~}~\textrm{for $y\subset Y$ and $O=\{O_y\}_{y\in Y}\in\prod_{y\in Y}\tau_y$}
\]
(also simply written as $[y,O]_p:=\{f\in\prod_yX_y:f(y)\in O\}$, for $y\subset Y$ and $O\in\tau_y$), hence with base consisting of sets of the form
\[
\textstyle [F,O]_p:=\bigcap_{y\in F}[y,O_y],~{~}~\textrm{for finite $F\subset Y$ and $O\in\prod_{y\in Y}\tau_y$.}
\]

Given a family of base spaces $\{X_y=(X_y,\tau_y,\mathcal{B}_y)\}_{y\in Y}$, where $\mathcal{B}_y=\bigcup_{\varepsilon_y\in\mathcal{E}_y}\mathcal{B}_{y,\varepsilon_y}$, for the product topology $\tau_p$, we pick the \textbf{$\tau_p$-product base space} $(\prod_{y\in Y} X_y,\tau_p,\mathcal{B}_p)$, $\mathcal{B}_p:=\bigcup_{\varepsilon\in\prod_y\mathcal{E}_y}\mathcal{B}_{p,\varepsilon}$, with the \emph{base cover} $\mathcal{B}_{p,\varepsilon}$ (where $\varepsilon=\{\varepsilon_y\}_{y\in Y}\in \prod_y\mathcal{E}_y$) given by finite intersections of the \emph{subbase cover}
\[
\textstyle \mathcal{S}\mathcal{B}_{p,\varepsilon}:=\big\{[y,O]_p~|~y\in Y,~O\in\prod_{y\in Y}\mathcal{B}_{y,\varepsilon_y}\big\}.
\]
More generally, with $\mathcal{Z}\subset\mathcal{P}(Y)$, for the topology $\tau_{uc|_{\mathcal{Z}}}$ on $\prod_yX_y$ (an update of that on $X^Y$ from  Definition \ref{TpUnfCnvDfn}) with subbase $\mathcal{S}\mathcal{B}_{uc|_{\mathcal{Z}}}$ consisting of the sets
{\small\[
\textstyle [Z,O]_{uc|_\mathcal{Z}}:=\{f\in \prod_yX_y:f(y)\in O_y~\forall y\in Z\},~{~}~\textrm{for}~{~}~Z\in\mathcal{Z},~O\in\prod_{y\in Y}\tau_y,
\]}we pick the \textbf{$\tau_{uc|_{\mathcal{Z}}}$-product base space} $(\prod_{y\in Y} X_y,\tau_{uc|_{\mathcal{Z}}},\mathcal{B}_{uc|_{\mathcal{Z}}})$, $\mathcal{B}_{uc|_{\mathcal{Z}}}:=\bigcup_{\varepsilon\in\prod_y\mathcal{E}_y}\mathcal{B}_{uc|_{\mathcal{Z}},\varepsilon}$, with the \emph{base cover} $\mathcal{B}_{uc|_{\mathcal{Z}},\varepsilon}$ given by finite intersections of the \emph{subbase cover}
\[
\textstyle \mathcal{S}\mathcal{B}_{uc|_{\mathcal{Z}},\varepsilon}:=\big\{[Z,O]_{uc|_{\mathcal{Z}}}~|~Z\in\mathcal{Z},~O\in\prod_{y\in Y}\mathcal{B}_{y,\varepsilon_y}\big\}.
\]

Let $Y$ be a set and $X$ a base space. A net $f_\alpha\in (X^Y,\tau_p,\mathcal{B}_p)$ is \textbf{pointwise cauchy} if $f_\alpha(y)\in X$ is cauchy for each $y\in Y$.
\end{dfn}

Henceforth, unless stated otherwise, $(\prod_{y\in Y} X_y,\tau_{uc|_{\mathcal{Z}}},\mathcal{B})$ denotes the above mentioned product base space (i.e., with $\mathcal{B}:=\mathcal{B}_{uc|_{\mathcal{Z}}}$).

\begin{dfn}[{Product $u$-structure, Product $u$-space}]\label{PrdUspDfn}
Let $\{(X_y,u_y)\}_{y\in Y}$ (where $u_y:X_y\times X_y\rightarrow (Z_y,\mathcal{E}_y),~(a,b)\mapsto u_y(a,b)$ is continuous) be a family of $u$-spaces. The associated \textbf{product $u$-space} $(X,u):=\big(\prod_{y\in Y}X_y,\prod_{y\in Y}u_y\big)$ is given by the \textbf{product $u$-structure}
\[
\textstyle u:X\times X\rightarrow (Z,\mathcal{E}):=\big(\prod_yZ_y,\mathcal{E}_{\prod_yZ_y}\big),~(x,x')\mapsto \big\{u_y(x(y),x'(y)\big\}_{y\in Y},
\]
where $\mathcal{E}:=\mathcal{E}_{\prod_{y\in Y}Z_y}$ is a ``desired collection'' of open sets in $Z:=\prod_{y\in Y}Z_y$ generated by the sets
\begin{align*}
 [F,\varepsilon]_p&\textstyle:=\{z\in Z:z(y)\in\varepsilon_y~\forall y\in F\}=\bigcap_{y\in F}[y,\varepsilon],~{}~\textrm{where $~[y,\varepsilon]:=\{z\in Z:z(y)\in \varepsilon_y\}$},
\end{align*}
for finite $F\subset Y$ and $\varepsilon=\{\varepsilon_y\}_{y\in Y}\in\mathcal{E}$.
({\bf footnote}\footnote{
Here, ``desired collection'' denotes ``uniform structure'' in the case where the $X_y$'s are uniform spaces.
}).
The associated \textbf{product $u$-balls} are
\begin{align*}
 B^u_{[F,\varepsilon]_p}(x)&\textstyle:=\big\{x'\in X:u(x,x')\in[F,\varepsilon]_p\big\}=\big\{x'\in X:u_y\big(x(y),x'(y)\big)\in\varepsilon_y~\forall y\in F\big\}
 \\
 &\textstyle=\bigcap_{y\in F}\big\{x'\in X:x'(y)\in B^{u_y}_{\varepsilon_y}(x(y))\big\}=\bigcap_{y\in F}B^u_{[y,\varepsilon]_p}(x)\\
 &\textstyle=\bigcap_{y\in F}[y,B^u_\varepsilon(x)],\\
 B^u_{[y,\varepsilon]_p}(x)&\textstyle:=[y,B^u_\varepsilon(x)]_p=[y,\{B^{u_y}_{\varepsilon_y}(x(y))\}_{y\in Y}]_p=\big\{x'\in X:x'(y)\in B^{u_y}_{\varepsilon_y}(x(y))\big\},
\end{align*}
where $~B^u_\varepsilon(x):=\{B^{u_y}_{\varepsilon_y}(x(y))\}_{y\in Y}\in\prod_{y\in Y}\tau_y$.

Therefore, the product $u$-space $(X,u)$ may be viewed as the \emph{natural} product base space $(X,\tau_u,\mathcal{B})$, $\mathcal{B}=\bigcup_{\varepsilon\in\prod_y\mathcal{E}_y}\mathcal{B}_\varepsilon$, with the \emph{base cover} $\mathcal{B}_\varepsilon$ given by finite intersections of the \emph{subbase cover}
\[
\textstyle \mathcal{S}\mathcal{B}_\varepsilon:=\big\{[y,B^u_\varepsilon(x)]_p~|~y\in Y\}.
\]
We note that, here, the \emph{natural} product base space $(X,\tau_u,\mathcal{B})$ associated with $(X,u)$ is just the \emph{$\tau_p$-product base space} from Definition \ref{AssocBsp}, which means $\tau_u=\tau_p$ and $\mathcal{B}=\mathcal{B}_p$. Consequently, $(X,u)$ also inherits the more general $\tau_{uc|_{\mathcal{Z}}}$-product base space structure $(X,\tau_{uc|_{\mathcal{Z}}},\mathcal{B}_{uc|_{\mathcal{Z}}})$ from Definition \ref{AssocBsp}, in which case the associated subbase of ball-like sets $\mathcal{S}\mathcal{B}_{uc|_{\mathcal{Z}}}$ (whose finite intersections give $\mathcal{B}_{uc|_{\mathcal{Z}}}$) is an update of $\mathcal{S}\mathcal{B}_{\textrm{ball}}$ from Remark \ref{BallTopRmk}. ({\bf footnote}\footnote{
Here, caution must be taken to avoid mistaking $Z=\prod_{y\in Y}Z_y$ for $Z\in\mathcal{Z}$ (which should now be changed to say $Z'\in\mathcal{Z}$).
})

For the special case where $X_y=(X_y,u_y):=(X_y,\mathcal{U}_y)$ is a uniform space, we have $Z_y=X_y\times X_y$, $u_y=id_{X_y\times X_y}$, and
\[
\varepsilon_y\in \mathcal{U}_y\subset\mathcal{P}(X_y\times X_y).
\]

\end{dfn}

\begin{lmm}\label{PtWsCompLmm}
Let $Y$ be a set, $\{X_y=(X_y,\tau_y,\mathcal{B}_y)\}_{y\in Y}$ a family of $lsb$-spaces, and $f_\alpha$, $g_\beta$ $\in$ $(\prod_y X_y,\tau_p,\mathcal{B})$ nets. Then $f_\alpha\rightarrowtail g_\beta$ if and only if $f_\alpha(y)\rightarrowtail g_\beta(y)$ in $X_y$ for each $y\in Y$.
\end{lmm}
\begin{proof}
($\Rightarrow$): Assume $f_\alpha\rightarrowtail g_\beta$. Suppose that for some $y\in Y$, $f_\alpha(y)\stackrel{}{\not\rightarrowtail} g_\beta(y)$ in $X_y$. Then there exist $\mathcal{B}_y$-neighborhoods
$O_{\varepsilon_{y,o}}(g_{\beta_b}(y))\in\mathcal{B}_{y,\varepsilon_{y,o}}$ that (for each $b$) exclude a subnet $f_{\alpha_{[a_{y,b},~]}}(y)$ of $f_\alpha(y)$, giving $\mathcal{B}$-neighborhoods $O_{\varepsilon_o}(g_{\beta_b}):=[y,O_{\varepsilon_{y,o}}(g_{\beta_b}(y))]_p\in\mathcal{B}_{\varepsilon_o}$ (where $\varepsilon_o(y):=\varepsilon_{y,o}$) in $\prod_yX_y$ that (for each $b$) exclude a subnet $f_{\alpha_{[a_{y,b},~]}}$ of $f_\alpha$ (a contradiction).

($\Leftarrow$): Assume $f_\alpha(y)\rightarrowtail g_\beta(y)$, in $X_y$, for all $y\in Y$. Suppose $f_\alpha\not\rightarrowtail g_\beta$. Then there exist $y_o\in Y$ and $\mathcal{B}$-neighborhoods
$O_{\varepsilon_o}(g_{\beta_b}):=[y_o,O_{\varepsilon_o(y_o)}(g_{\beta_b}(y_o))]_p\in\mathcal{B}_{\varepsilon_o}$ in $\prod_yX_y$ that (for each $b$) exclude a subnet $f_{\alpha_{[a_{y_o,b},~]}}$ of $f_\alpha$ (i.e., $f_{\alpha_{[a_{y_o,b},~]}}(y_o)\subset X_{y_o}\backslash O_{\varepsilon_o(y_o)}(g_{\beta_b}(y_o))$). Since $f_\alpha(y_o)\rightarrowtail g_\beta(y_o)$, for every $O_{\varepsilon_{y_o}}(g_\beta(y_o))\in\mathcal{B}_{y_o,\varepsilon_{y_o}}$, there exists $\beta^{\{O_{\varepsilon_{y_o}}(g_\beta(y_o))\}_\beta}$ such that for all $\beta\geq \beta^{\{O_{\varepsilon_{y_o}}(g_\beta(y_o))\}_\beta}$, $O_{\varepsilon_{y_o}}(g_\beta(y_o))$ contains a tail of $f_\alpha(y_o)$. So, with $\varepsilon_{y_o}:=\varepsilon_o(y_o)$ and $\beta:=\beta_b$, we get a contradiction.
\end{proof}

\begin{thm}\label{PtWsCompThm}
Let $\{(X_y,\tau_y,\mathcal{B}_y))\}_{y\in Y}$, $(X,\tau_X,\mathcal{B}_X)$, and $(X^Y,\tau_p,\mathcal{B})$ be $lsb$-spaces.

(i) A net in $(X^Y,\tau_p,\mathcal{B})$ is cauchy if and only if pointwise cauchy.

(ii) $(\prod_yX_y,\tau_p,\mathcal{B})$ is complete if and only if each $X_y=(X_y,\tau_y,\mathcal{B}_y)$ is complete.

(iii) $(\widetilde{\prod X_y},\tau_p,\mathcal{B})\cong\prod(\widetilde{X_y},\tau_y,\mathcal{B}_y)$,~ where $\widetilde{X}$ denotes the completion of $X$.

(iv) $(\prod_yX_y,\tau_p,\mathcal{B})$ is precompact if and only if each $X_y=(X_y,\tau_y,\mathcal{B}_y)$ is precompact.
\end{thm}
\begin{proof}
(i)-(iii) are immediate (by Lemma \ref{PtWsCompLmm}, the fact each $X_y$ is a closed subspace of $(\prod_yX_y,\tau_p,\mathcal{B})$, and the fact we have a dense uniform embedding $\prod_yX_y\hookrightarrow\prod_y\widetilde{X_y}$, as $\overline{\prod_yX_y}=\prod_y\overline{X_y}$ in $\prod_y\widetilde{X_y}$). So we prove (iv): By Tychonoff's product theorem, $(\prod_yX_y,\tau_p,\mathcal{B})$ is compact if and only of each $X_y=(X_y,\tau_y,\mathcal{B}_y)$ is compact. Also, $(\prod_yX_y,\tau_p,\mathcal{B})$ is precompact if and only if it has a compact completion (if and only if, by (iii), each $(X_y,\tau_y,\mathcal{B}_y)$ has a compact completion.)
\end{proof}

\begin{lmm}\label{FxnNetApp1}
Let $Y$ be a set, $(X,\tau_X,\mathcal{B}_X)$ and $(X^Y,\tau_{uc},\mathcal{B})$ $lsb$-spaces, and $f_\alpha,g_\beta\in (X^Y,\tau_{uc},\mathcal{B})$ nets. If $f_\alpha\rightarrowtail g_\beta$, then $f_\alpha(y)\rightarrowtail g_\beta(y)$ in $X$, for all $y\in Y$.
\end{lmm}
\begin{proof}
Recall that $\tau_p\subset\tau_{uc}$. Therefore, the rest is just the forward direction $(\Rightarrow)$ of the proof of Lemma \ref{PtWsCompLmm}.

\begin{comment}
Assume $f_\alpha\rightarrowtail g_\beta$. Suppose that for some $y\in Y$, $f_\alpha(y)\stackrel{}{\not\rightarrowtail} g_\beta(y)$ in $X$. Then there exist $\mathcal{B}_X$-neighborhoods
$O_{\varepsilon_o}(g_{\beta_b}(y))\in\mathcal{B}_{X,\varepsilon_o}$ that (for each $b$) exclude a subnet $f_{\alpha_{[a_{y,b},~]}}(y)$ of $f_\alpha(y)$, giving $\mathcal{B}$-neighborhoods $O_\varepsilon(g_{\beta_b}):=[y,O_{\varepsilon_o}(g_{\beta_b}(y))]_p\in\mathcal{B}_{\varepsilon_o}$ in $X^Y$ that (for each $b$) exclude a subnet $f_{\alpha_{[a_{y,b},~]}}$ of $f_\alpha$ (a contradiction).
\end{comment}
\end{proof}

\begin{comment}
\begin{proof}
(i)-(iii) are immediate (by Lemma \ref{PtWsCompLmm}, the fact each $X_y$ is a closed subspace of $(\prod_yX_y,\tau_p,\mathcal{B})$, and the fact we have a dense uniform embedding $\prod_yX_y\hookrightarrow\prod_y\widetilde{X_y}$, as $\overline{\prod_yX_y}=\prod_y\overline{X_y}$ in $\prod_y\widetilde{X_y}$). So we prove (iv): By Tychonoff's product theorem, $(\prod_yX_y,\tau_p,\mathcal{B})$ is compact if and only of each $X_y=(X_y,\tau_y,\mathcal{B}_y)$ is compact. Also, $(\prod_yX_y,\tau_p,\mathcal{B})$ is precompact if and only if it has a compact completion (if and only if, by (iii), each $(X_y,\tau_y,\mathcal{B}_y)$ has a compact completion.)
\end{proof}
\end{comment}

\begin{crl}\label{PtWTcchCrl}
Let $Y$ be a set and $(X,\tau_X,\mathcal{B}_X$) and $(X^Y,\tau_{uc},\mathcal{B})$ $lsb$-spaces. If a net in $(X^Y,\tau_{uc},\mathcal{B})$ is cauchy (resp., convergent), then it is pointwise cauchy (resp., pointwise convergent).
\end{crl}

\begin{thm}[{Completeness of $\tau_{uc}$}]\label{UnfCmplThm}
If $Y$ is a set and $X=(X,\tau_X,\mathcal{B}_X)$ a complete $lsb$-space, then $(X^Y,\tau_{uc},\mathcal{B})$ is a complete $lsb$-space. (The converse also holds if $X$ is Hausdorff, i.e., if $X=(X,\tau_X)$ is Hausdorff, then $X=(X,\tau_X,\mathcal{B}_X)$ is complete iff $(X^Y,\tau_{uc},\mathcal{B})$ is complete.)
\end{thm}
\begin{proof}
Let $f_\alpha\in (X^Y,\tau_{uc})$ be cauchy. Then, by Corollary \ref{PtWTcchCrl}, $f_\alpha(y)$ is cauchy in $X$, for all $y\in Y$. For each $y\in Y$, let $f_\alpha(y)\rightarrow f(y)$. Then $f_\alpha\stackrel{u}{\longrightarrow} f$ in $X^Y$. Otherwise, suppose $f_\alpha\stackrel{u}{\not\rightarrow} f$ in $X^Y$, i.e., by cauchyness, some $\mathcal{B}$-neighborhood $N(f)$ of $f$ excludes a tail $f_{[\alpha_0,~]}$ of $f_\alpha$. However, for each $y\in Y$, we can pick a $\mathcal{B}_X$-neighborhood $N(f(y))$, of $f(y)$, and $\alpha_y\geq\alpha_0$ such that $f_{[\alpha_y,~]}(y)\subset N(f(y))$, in which case
\[
f_{[\alpha_y,~]}\cap N_y(f)\neq\emptyset,
\]
where ~$N_y(f):=N(f)\cap[y,N(f(y))]_p=\{h\in N(f):h(y)\in N(f(y))\}$.

So, we can pick elements of $f_{[\alpha_0,~]}$ of the form $f_{\alpha_{a(y)}}\in N_y(f)\subset N(f)$, a contradiction. NB: Since $\tau_p\subset\tau_{uc}$, we have $N_y(f)\in\tau_{uc}$, and so it is not essential to know the explicit form of $N(f)$.

(The converse, for a Hausdorff $X$, follows because the subspace $X\cong Const(Y,X)\subset (X^Y,\tau_{uc})$ of constant maps is closed. Indeed, let $c_x:Y\rightarrow X,~y\mapsto x$ and pick a net $c_{x_\alpha}\in Const(Y,X)$ such that $c_{x_\alpha}\stackrel{u}{\longrightarrow} f\in X^Y$ (which implies $c_{x_\alpha}\stackrel{\tau_p}{\longrightarrow} f$ since $\tau_p\subset\tau_{uc}$). Then for any family of open sets $\mathcal{O}_\varepsilon=\{O_y\ni f(y)\}_{y\in Y}\subset \mathcal{B}_{X,\varepsilon}$ in $X$, there exists $\alpha^{\mathcal{O}_\varepsilon}$ such that $c_{x_{[\alpha^{\mathcal{O}_\varepsilon},~]}}(y)=x_{[\alpha^{\mathcal{O}_\varepsilon},~]}\subset O_y$ for all $y\in Y$, i.e., $x_\alpha\rightarrow f(y)$ for all $y\in Y$, and so $f$ must be constant. Although not of immediate relevance to us, it is also clear that $X\cong Const(Y,X)\subset (X^Y,\tau_p)$ is closed.)
\end{proof}

\begin{lmm}\label{ContUlmLmm} Lemma \ref{ContUlmLmm}
Let $Y=(Y,\tau_Y,\mathcal{B}_Y)$ be a $lsb$-space, $X=(X,\tau_X,\mathcal{B}_X)$ a complete $lsb$-space, and $f\in X^Y$.

(i) Let $f_\bullet\in C(Y,X)\subset (X^Y,\tau_{uc},\mathcal{B})$ such that $f_\bullet\stackrel{u}{\longrightarrow}f$. If $y_\bullet\rightarrow y_0$ in $Y$, then  $f_\bullet(y_\bullet)\rightarrowtail f_\bullet(y_0)$ in $X$.

(ii) Let $f_\bullet\in UC(Y,X)\subset (X^Y,\tau_{uc},\mathcal{B})$ such that $f_\bullet\stackrel{u}{\longrightarrow}f$. If $y_\bullet\rightarrowtail\widetilde{y}_\bullet$ in $Y$, then  $f_\bullet(y_\bullet)\rightarrowtail f_\bullet(\widetilde{y}_\bullet)$ in $X$.
\end{lmm}

\begin{proof}
In both (i) and (ii), since $f_\bullet\stackrel{u}{\longrightarrow}f$, $f_\bullet$ is cauchy (hence $f_\bullet\rightarrowtail f_\bullet$) in $(X^Y,\tau_{uc})$. Since $\tau_p\subset\tau_{uc}$:

(i) For any $\mathcal{O}_\varepsilon(f_\bullet(y_0))=\{O_\varepsilon(f_\alpha(y_0))\ni f_\alpha(y_0)\}\subset\mathcal{B}_{\varepsilon,X}$ (or equivalently, for $\mathcal{O}_\varepsilon(f_\bullet):=\{O_\varepsilon(f_\alpha)=[y_0,O_\varepsilon(f_\alpha(y_0))]_p:\forall\alpha\}\subset\mathcal{B}_\varepsilon$), there exists $\alpha^{\mathcal{O}_\varepsilon(f_\bullet)}$ such that $\forall\alpha\geq \alpha^{\mathcal{O}_\varepsilon(f_\bullet)}$, $[y_0,O_\varepsilon(f_\alpha(y_0))]_p$ contains a tail of $f_\bullet$ (i.e., $O_\varepsilon(f_\alpha(y_0))$ contains a tail of $f_\bullet(y_0)$, hence because each $f_\alpha(y_\bullet)\rightarrow f_\alpha(y_0)$, also contains a tail of $f_\bullet(y_\bullet)$). Therefore $f_\bullet(y_\bullet)\rightarrowtail f_\bullet(y_0)$.

(ii) For any $\mathcal{O}_\varepsilon(f_\bullet(\widetilde{y}_\bullet))=\{O_\varepsilon(f_\alpha(\widetilde{y}_{\widetilde{\beta}}))\ni f_\alpha(\widetilde{y}_{\widetilde{\beta}})\}\subset\mathcal{B}_{\varepsilon,X}$ (or equivalently, for $\mathcal{O}_\varepsilon(f_\bullet):=\{\mathcal{O}_\varepsilon(f_\bullet(\widetilde{y}_{\widetilde{\beta}})):\forall\widetilde{\beta}\}=\{[\widetilde{y}_{\widetilde{\beta}},O_\varepsilon(f_\alpha(\widetilde{y}_{\widetilde{\beta}}))]_p:\forall\alpha,\widetilde{y}_{\widetilde{\beta}}\}\subset\mathcal{B}_\varepsilon$), there exist $\{\alpha_{\widetilde{\beta}}^{\mathcal{O}_\varepsilon(f_\bullet)}:\forall \widetilde{\beta}\}$ such that $\forall\alpha\geq \alpha_{\widetilde{\beta}}^{\mathcal{O}_\varepsilon(f_\bullet)}$, $[\widetilde{y}_{\widetilde{\beta}},O_\varepsilon(f_\alpha(\widetilde{y}_{\widetilde{\beta}}))]_p$ contains a tail of $f_\bullet$ (i.e., for each $\widetilde{\beta}$, $O_\varepsilon(f_\alpha(\widetilde{y}_{\widetilde{\beta}}))$ contains a tail of $f_\bullet(\widetilde{y}_{\widetilde{\beta}})$, hence because each $f_\alpha(y_\bullet)\rightarrowtail f_\alpha(\widetilde{y}_\bullet)$, there exists $\widetilde{\beta}'$ (hence $\alpha^{\mathcal{O}_\varepsilon(f_\bullet)}:=\alpha_{\widetilde{\beta}'}^{\mathcal{O}_\varepsilon(f_\bullet)}$) such that $\forall\widetilde{\beta}\geq\widetilde{\beta}'$ (and $\alpha\geq \alpha^{\mathcal{O}_\varepsilon(f_\bullet)}$), $O_\varepsilon(f_\alpha(\widetilde{y}_{\widetilde{\beta}}))$ also contains a tail of $f_\bullet(y_\bullet)$). Therefore $f_\bullet(y_\bullet)\rightarrowtail f_\bullet(\widetilde{y}_\bullet)$.
\end{proof}

\begin{thm}[{Continuous limit theorem, Uniform limit theorem}]\label{ContUlmThm}
Let $Y=(Y,\tau_Y,\mathcal{B}_Y)$ be a $lsb$-space and $X=(X,\tau_X,\mathcal{B}_X)$ a complete $lsb$-space. The following hold:
\begin{enumerate}[leftmargin=0.7cm]
\item[(i)] The space of continuous maps $C(Y,X)\subset(X^Y,\tau_{uc},\mathcal{B})$ is complete.
\item[(ii)] The space of uniform maps $UC(Y,X)\subset (X^Y,\tau_{uc},\mathcal{B})$) is complete.
\end{enumerate}

\end{thm}
\begin{proof}
By Theorem \ref{UnfCmplThm}, $(X^Y,\tau_{uc},\mathcal{B})$ is a complete $lsb$-space. So, it suffices to prove that ~$C(Y,X)$ and $UC(Y,X)$ are closed subsets of $(X^Y,\tau_{uc},\mathcal{B})$.

(i) Let $f_\bullet \in C(Y,X)$ and $f\in X^Y$ such that $f_\bullet\stackrel{u}{\longrightarrow}f$. We need to show $f\in C(Y,X)$. By Corollary \ref{PtWTcchCrl}, $f_\bullet(y)\rightarrow f(y)$ for all $y\in Y$. Let $y_\bullet\rightarrow y_0$ in $Y$. Then by Lemmas \ref{ContUlmLmm}(i) and \ref{BcnvLimit}, $f(y_\bullet)\rightarrow f(y_0)$, since

\begin{center}
\begin{tikzcd}
f_\bullet(y_\bullet)\ar[d,tail] & \forall\beta,~~~f_\bullet(y_\beta)\ar[rr,tail]  && f(y_\beta) & f(y_\bullet) \ar[d,dashed,tail]\\
f_\bullet(y_0)            & f_\bullet(y_0)   \ar[rr,tail]  && f(y_0) & f(y_0)
\end{tikzcd}
\end{center}

(ii) Let $f_\bullet\in UC(Y,X)$ and $f\in X^Y$ such that $f_\bullet\stackrel{u}{\longrightarrow}f$. We need to show $f\in UC(Y,X)$. By Corollary \ref{PtWTcchCrl}, $f_\bullet(y)\rightarrow f(y)$ for all $y\in Y$. Let $y_\bullet\rightarrowtail\widetilde{y}_\bullet$ in $Y$. Then by Lemmas \ref{ContUlmLmm}(ii) and \ref{BcnvLimit}, $f(y_\bullet)\rightarrowtail f(\widetilde{y}_\bullet)$, since

\begin{center}
\begin{tikzcd}
f_\bullet(y_\bullet)\ar[d,tail] & \forall\beta,~~~f_\bullet(y_\beta)\ar[rr,tail]  && f(y_\beta) & f(y_\bullet) \ar[d,dashed,tail]\\
f_\bullet(\widetilde{y}_\bullet)            & \forall\widetilde{\beta},~~~f_\bullet(\widetilde{y}_{\widetilde{\beta}})   \ar[rr,tail]  && f(\widetilde{y}_{\widetilde{\beta}}) & f(\widetilde{y}_\bullet)
\end{tikzcd}
\end{center}
\end{proof}

\begin{rmk}
The results Lemma \ref{FxnNetApp1}, Corollary \ref{PtWTcchCrl}, Theorem \ref{UnfCmplThm},  Lemma \ref{ContUlmLmm} Theorem \ref{ContUlmThm} remain true if $\tau_{uc}$ is replaced with $\tau_{uc|_{\mathcal{Z}}}$ (from Definition \ref{TpUnfCnvDfn}) or any other topology $\tau\supset\tau_p$.
\end{rmk}

%%#######################################################################################
\section{\textnormal{\bf Integration in topological modules}}\label{IntModSec}
%%#######################################################################################
\begin{dfn}[{Topological Group, Ring, Module}]
A \textbf{topological group} is a group $G=(G,\cdot,()^{-1},e)$ that is a topological space with continuous group-multiplication and inversion
\[
\cdot:G\times G\rightarrow G,~(g,h)\mapsto gh~~{~}~~\textrm{and}~~{~}~~()^{-1}:G\rightarrow G,~g\mapsto g^{-1}.
\]
A \textbf{topological ring} $R=(R,+,\cdot,0,1)$ is a (topological abelian group endoured with a continuous multiplication operation, making it precisely/equivalently a) ring that is a topological space with continuous addition and multiplication
\[
~~~~+:R\rightarrow R,~(r,s)\mapsto r+s~~{~}~~\textrm{and}~~{~}~~\cdot:R\times R\rightarrow R,~(r,s)\mapsto rs.
\]
Let $R$ be a topological ring. A \textbf{topological $R$-module} $M={}_RM=(M,+,\cdot_R,0)$ is an $R$-module that is a topological space with continuous addition and scalar multiplication
\[
+:M\rightarrow M,~(m,n)\mapsto m+n~~{~}~~\textrm{and}~~{~}~~\cdot_R:R\times M\rightarrow M,~(r,m)\mapsto rm.
\]
\end{dfn}
In concrete applications, we often take ~$R=\mathbb{F}\in\{\mathbb{Z}_p,\mathbb{R},\mathbb{C}\}$, where $\mathbb{Z}_p:=\mathbb{Z}/p\mathbb{Z}$ for a prime number $p$, which can also be accounted for by considering product rings of the form $R:=\mathbb{Z}_p{}^m\times\mathbb{R}^n\times\mathbb{C}^k$, for $m,n,k\in\mathbb{N}$.

\begin{dfn}[{Measure, Integral}]\label{GenMeasInt}
Let $R$ be a topological ring, $U$ a set, and $X$ a \textbf{$sb$-topological $R$-module} (i.e., a topological $R$-module as a $lsb$-space). An \textbf{$X$-valued measure} on $U$ is a map
\[
\mu:\Sigma\subset\mathcal{P}(U)\rightarrow X,~\Delta\mapsto \mu(\Delta),~~\textrm{~}~~\textrm{$\Sigma$ a $\sigma$-algebra} ~~(\textrm{\bf footnote}\footnotemark),
\]
\footnotetext{A subset of $\mathcal{P}(U)$ is a \textbf{$\sigma$-algebra} on $U$ if it contains $U$ and is closed under both complement and countable union.}such that (i) $\mu(\emptyset)=0$ and (ii) for any countable disjoint family $\{\Delta_i\}_{i=1}^\infty\subset\Sigma$, we have $\mu(\bigcup_i\Delta_i)=\sum_{i=1}^\infty\mu(\Delta_i):=\lim_{\textrm{finite $S\subset\mathbb{N}$}}\sum_{i\in S}\mu(\Delta_i)$.

Consider the set of $\Sigma$-measurable \emph{indexed} finite partitions of $U$,
\[
\textstyle \mathbb{M}_{\textrm{FP}}(U, \Sigma) :=\{(F,\mathcal{P}):\mathcal{P}\subset \mathcal{P}(\Sigma)~\textrm{a finite partition of $U$},~F\in\prod\mathcal{P}~\textrm{an indexing by members of $\mathcal{P}$}\},
\]
as the poset ({\bf footnote}\footnote{
This is a directed set by construction, where an upper bound of $(F,\mathcal{P})$ and $(G,\mathcal{Q})$ is $(F\cup G,\mathcal{P}\cup\mathcal{Q})$.
}) with ordering ``$(F,\mathcal{P}) \leq  (F',\mathcal{P}')$ if $(\mathcal{P}',F')$ is a \textbf{refinement} of $(F,\mathcal{P})$ in that $F\subset F'$ and for each $P'\in\mathcal{P}'$, there exists $P\in \mathcal{P}$ such that $P'\subset P$''. In $(F,\mathcal{P})$, if $t\in F$, let $P_t$ denote the lone element of $\mathcal{P}$ satisfying $F\cap P_t=\{t\}$ (so that $\mathcal{P}=\{P_t\}_{t\in F}\equiv P_F$).

Given a map $f:U\rightarrow R$, associate the net
\[
\textstyle S_f=S_f(U,\Sigma,\mu):\mathbb{M}_{\textrm{FP}}(U, \Sigma)\rightarrow X,~(F,\mathcal{P})\mapsto\sum_{t\in F}f(t)\mu( P_t)\in\textrm{Span}_R(\mu(\Sigma)).~~(\textrm{\bf footnote}\footnotemark).
\]
\footnotetext{For example, if $f:(U,\Sigma)\rightarrow(R,\Sigma_R)$ is measurable (i.e., $f^{-1}(\Sigma_R)\subset\Sigma$), then we may set $P_F:=f^{-1}(Q_F)=\{f^{-1}(Q_t):t\in F\}$ for a $\Sigma_R$-measurable partition $Q_F$ of $f(U)$. \textbf{Alternatively}, consider the set $I:=\{(p,\mathcal{P}):\mathcal{P}\subset\Sigma~\textrm{a finite set},~p:\Sigma\rightarrow U,~\Delta\mapsto p(\Delta)\in\Delta~\textrm{a selection map}\}$ as the directed set with ordering ``$(p,\mathcal{P})\leq (p',\mathcal{P}')$ if $\mathcal{P}\subset \mathcal{P}'$''. Given a function $f:U\rightarrow R$, consider the net
\[
\textstyle S_f:I\rightarrow X,~(p,\mathcal{P})\mapsto\sum_{\Delta\in \mathcal{P}}f(p(\Delta))\mu(\Delta).
\]}An \textbf{$X$-valued integral} of $f$ with respect to $\mu$ is any limit of $S_f$ (denoted by $\int_Ufd\mu$ or $\int_Uf(t)d\mu_t$ in $\overline{\textrm{Span}_R(\mu(\Sigma))}$).
\end{dfn}

\begin{rmk}
If $X$ is a complete Hausdorff $lsb$-space, then the above integral (whenever it exists) is unique, and moreover, existence of the integral is equivalent to cauchyness of $S_f$. This gives an alternative method of integration based on our notion of cauchyness in Definition \ref{ApprchDfn}.
\end{rmk}

%%#######################################################################################
\section{\textnormal{\bf Conclusion and questions}}\label{ConclSec}
%%#######################################################################################
We have seen that an induction-free concept of a cauchy net (hence of completeness) is possible in a base space (a topological space with a graded base). Our discussion is based on a notion of approach between nets and reproduces many classical results on completeness for uniform spaces and associated product and function spaces. Our method of completeness provides a means of integration in topological $R$-modules.

The following are more research questions (in addition to Questions \ref{AprDfnQsn}, \ref{BodyQuest2}, \ref{BodyQuest3}).

\begin{question}
Which subset $lsb$-hyperspaces of a complete $lsb$-space are complete?
\end{question}

\begin{question}
We know that every uniform space is a $csb$-space (by Proposition \ref{UnifCsbPrp}). (i) Is every $csb$-space uniformizable? (ii) Is every $lsb$-space a $u$-space?
\end{question}

\begin{question}
To what extent can our main results for $lsb$-spaces (Theorems \ref{CompCsbThm}, \ref{CpctSpThm}, \ref{BaireThm}, \ref{tUConThm}, \ref{BSpCmplThm}, \ref{EBPCSthm}, \ref{PtWsCompThm}, \ref{UnfCmplThm}, \ref{ContUlmThm}) be extended to more general base spaces?
\end{question}

\begin{question}
Let us say that a $lsb$-space is \textbf{uniformly connected} if every two points in the space are linked by a uniform path. Do the main results of \cite{akofor2024} still hold if ``Lipschitz paths'' are replaced with ``uniform paths''? This is relevant to \cite[Question 5.1]{akofor2024}, which was partly addressed in \cite[Question 5]{akofor2025} for the case of finite-subset hyperspaces of spaces more general than metric spaces.
\end{question}

\begin{question}[Completeness of (topological) categories]
Can the ``net-approach'' idea be directly applied to ``convergence systems'' in category theory \cite[Ch. 9]{akofor2021}, by replacing the ``system to object convergence structure'' with a ``system to system approach structure''? In abstract settings, one might need to prove the existence of appropriate ``approach structures'' using Zorn's lemma (beginning with the convergence itself as an approach structure on constant/trivial systems as the base case for the induction, and then extending to nonconstrant/nontrivial systems).
\end{question}

\vspace{0.2cm}
\hrule
%%%%%%%%%%%%%%%%%%%%%%%%%%%%%%%%%%%%%%%%%%%%%%%%%%%%%%%%%%%%%%%%%%%%%%%%%%%
%\endgroup

\begin{bibdiv}
\begin{biblist}
\bib{AarAnd1972}{article}{  %%% [1]
   author={Aarnes, J. F.,},
   author={Andenes, P. R.,},
%   author={},
   title={On nets and filters},
   note={},
   journal={Math. Scand.},
   volume={31},
   date={1972},
%   number={},
   pages={285-292},
   issn={},
%%%   review={\mathbb{M}R{3563748}},
   doi={},
}

\bib{akofor2025}{article}{  %%% [1]
   author={Akofor, E.},
%   author={Jafari, S.},
%   author={},
   title={Modulo arithmetic of function spaces: Subset hyperspaces as quotients of function spaces},
   note={},
   journal={Houston Journal of Mathematics},
   volume={},
   date={2025},
   number={},
%   pages={89-100},
   issn={},
%%%   review={\mathbb{M}R{3563748}},
   doi={},
}

\bib{akofor2024}{article}{
   author={Akofor, E.},
%   author={},
   title={On quasiconvexity of precompact-subset spaces},
%   note={Submitted to ``Ann. Fenn. Math.''},
   journal={J. Anal.},
%   volume={},
   date={2024},
%   number={},
%   pages={},
%%%   issn={},
%%%   review={},
  doi={\url{10.1007/s41478-024-00827-z}},
}

\bib{akofor25svm}{article}{  %%% [1]
   author={Akofor, E.},
%   author={Jafari, S.},
%   author={},
   title={Set-valued metrics and generaized Hausdorff distances},
   note={},
%   journal={arXiv:2303.10815 [math.GN]},
   volume={},
   date={2025},
   number={},
%   pages={89-100},
   issn={},
%%%   review={\mathbb{M}R{3563748}},
   doi={\url{https://doi.org/10.48550/arXiv.2503.10815}},
}

\bib{akofor2021}{article}{  %%% [1]
   author={Akofor, E.},
%   author={Jafari, S.},
%   author={},
   title={Basic Set Theory and Algebra: Hints on Representation, Topology, Geometry, Analysis},
   note={},
%   journal={arXiv:2101.02031 [math.CT]},
   volume={},
   date={2021},
   number={},
%   pages={89-100},
   issn={},
%%%   review={\mathbb{M}R{}},
   doi={\url{https://doi.org/10.48550/arXiv.2101.02031}},
}

\begin{comment}
\bib{Bal2000}{book}{
   author={Balachandran, V. K.,},
   title={Topological algebras},
   series={Mathematics Studies},
   volume={185},
   note={},
   publisher={Elsevier},
   date={2000},
   pages={},
   isbn={},
%   review={\mathbb{M}R{1670250}},
}
\end{comment}

\bib{Bartle1955}{article}{  %%% [1]
   author={Bartle, R. G.,},
%   author={},
   title={Nets and filters in topology},
   note={},
   journal={Amer. Math. Month.},
   volume={62},
   date={1955},
   number={8},
   pages={551-557},
   issn={},
%%%   review={\mathbb{M}R{3563748}},
%   doi={https://doi.org/10.1080/00029890.1955.11988688},
}

\bib{BentEtal1990}{article}{  %%% [1]
   author={Bentley, H. L.},
   author={Herrlich, H.},
   author={Lowen-Colebunders, E.},
   title={Convergence},
   note={},
   journal={Journ. Pure App. Alg.},
   volume={68},
   date={1990},
%   number={},
   pages={27-45},
   issn={},
%%%   review={\mathbb{M}R{3563748}},
   doi={},
}

\bib{BbkGT1966}{book}{
   author={Bourbaki, N.,},
   title={General Topology I $\&$ II},
   series={Elements of Mathematics},
   volume={},
   note={},
   publisher={Hermann},
   date={1966},
   pages={},
   isbn={},
%   review={\mathbb{M}R{1670250}},
}

\bib{Choquet47-48}{article}{  %%% [1]
   author={Choquet, G.},
%   author={},
   title={Convergences},
   note={},
   journal={Annales de l'universit\'e de Grenoble},
   volume={23},
   date={1947},
%   number={},
   pages={57-112},
   issn={},
%%%   review={\mathbb{M}R{3563748}},
   doi={},
}

\bib{Dolecki2024}{book}{
   author={Dolecki, S.,},
%%   author={Mynard, F.,},
   title={A Royal Road to Topology: Convergence of Filters},
   series={},
   volume={},
   note={},
   publisher={World Scientific Publishing},
   date={2024},
   pages={},
   isbn={},
%   review={\mathbb{M}R{1670250}},
}

\bib{DolMyn2016}{book}{
   author={Dolecki, S.,},
   author={Mynard, F.,},
   title={Convergence Foundations of Topology},
   series={},
   volume={},
   note={},
   publisher={World Scientific Publishing},
   date={2016},
   pages={},
   isbn={},
%   review={\mathbb{M}R{1670250}},
}

\bib{engelking1989}{book}{
   author={Engelking, R.,},
   title={General Topology},
   series={Sigma series in pure mathematics; Vol. 6},
   volume={},
   note={},
   publisher={Heldermann Verlag Berlin},
   date={1989},
   pages={},
   isbn={3-88538-Q06-4},
%   review={\mathbb{M}R{1670250}},
}

\bib{Freiwald2014}{book}{
   author={Freiwald, R. C.,},
   title={An Introduction to Set Theory and Topology},
   series={},
   volume={},
   note={},
   publisher={http://doi.dx.org/10.7936/K7D798QH},
   date={2014},
   pages={},
   isbn={978-1-941823-10-1},
%   review={\mathbb{M}R{1670250}},
   doi={10.7936/K7D798QH},
}

\bib{Isbell1964}{book}{
   author={Isbell, J. R.,},
   title={Uniform spaces},
   series={Mathematical Surveys 12},
   volume={},
   note={},
   publisher={Amer. Math. Soc.},
   date={1964},
   pages={},
   isbn={},
%   review={\mathbb{M}R{1670250}},
}

\bib{James1987}{book}{
   author={James, I. M.,},
   title={Topological and Uniform Spaces},
   series={Undergraduate Texts in Mathematics},
   volume={},
   note={},
   publisher={Springer-Verlag New York Inc.},
   date={1987},
   pages={},
   isbn={},
%   review={\mathbb{M}R{1670250}},
}

\bib{keller1968}{article}{  %%% [1]
   author={Keller, H. H.,},
%   author={},
   title={Die Limes-Uniformisierbarkeit der Limesr\"aume},
   note={},
   journal={Math. Ann.},
   volume={176},
   date={1968},
%   number={},
   pages={334-341},
   issn={},
%%%   review={},
   doi={},
}

\bib{kelley1955}{book}{
   author={Kelley, J. K.,},
   title={General Topology},
   series={University Series in Higher Mathematics},
   volume={},
   note={},
   publisher={D. Van Nostrand Company},
   date={1955},
   pages={},
   isbn={},
%   review={\mathbb{M}R{1670250}},
}

\bib{Lowen1989}{book}{
   author={Lowen-Colebunders, E.},
%   author={Warrack, B. D.},
   title={Function Classes of Cauchy Continuous Maps},
   series={Pure and Applied Mathematics},
   volume={},
   note={},
   publisher={Marcel Dekker},
   date={1989},
   pages={},
   isbn={},
%   review={\mathbb{M}R{1670250}},
}

\bib{Munkres2000}{book}{
   author={Munkres, J. R.},
   title={Topology, 2nd Edition},
   series={},
   volume={},
   note={},
   publisher={Prentice Hall},
   date={2000},
   pages={},
   isbn={},
%   review={\mathbb{M}R{1670250}},
}

\bib{NmpWrrk1970}{book}{
   author={Naimpally, S. A.,},
   author={Warrack, B. D.,},
   title={Proximity Spaces},
   series={Cambridge Tracts in Mathematics and Mathematical Physics, No. 59},
   volume={},
   note={},
   publisher={Cambridge University Press},
   date={1970},
   pages={},
   isbn={},
%   review={\mathbb{M}R{1670250}},
}

\bib{Nakano1977}{article}{  %%% [1]
   author={Nakano, K.,},
%   author={},
   title={A notion of completeness of topological structures},
   note={Paper read at the Second Symposium on Categorical Topology at the University of Cape Town 9-13 August 1976},
   journal={Quaestiones Mathematicae},
   volume={2},
   date={1977},
%   number={},
   pages={235-243},
   issn={},
%%%   review={\mathbb{M}R{3563748}},
   doi={},
}

\bib{neumann1935}{article}{  %%% [1]
   author={Neumann, J. V.,},
%   author={},
   title={On complete topological spaces},
   note={},
   journal={Trans. Amer. Math. Soc.},
   volume={37},
   date={Jan., 1935},
   number={1},
   pages={1-20},
   issn={},
%%%   review={\mathbb{M}R{3563748}},
   doi={},
}

\bib{ObrienEtal2021}{article}{  %%% [1]
   author={O'Brien, M.,},
   author={Troitsky, V. G.,},
   author={Van Der Walt, J. H. ,},
%   author={},
   title={Net convergence structures with applications to vector lattices},
   note={},
   journal={arXiv:2103.01339v1 [math.FA]},
   volume={},
   date={2021},
%   number={},
   pages={},
   issn={},
%%%   review={\mathbb{M}R{3563748}},
   doi={},
}

\bib{OBrien2021}{article}{  %%% [1]
   author={O'Brien, M.,},
%   author={},
   title={A theory of net convergence with applications to vector lattices},
   note={},
   journal={Ph.D. Thesis, University of Alberta},
   volume={},
   date={2021},
%   number={},
   pages={},
   issn={},
%%%   review={\mathbb{M}R{3563748}},
   doi={},
}

\bib{Schect1997}{book}{
   author={Schechter, E.},
%   author={Warrack, B. D.},
   title={Handbook of Analysis and its Foundations},
   series={},
   volume={},
   note={},
   publisher={Academic Press},
   date={1997},
   pages={},
   isbn={},
%   review={\mathbb{M}R{1670250}},
}

%\end{comment}

\end{biblist}
\end{bibdiv}
\end{document}